\documentclass[psamsfonts]{amsart}

\usepackage{amsmath,amssymb,amsthm,enumerate,tikz-cd,comment,url,thmtools,thm-restate,diagrams,url}

\numberwithin{equation}{section}

\theoremstyle{plain}
\newtheorem{theorem}{Theorem}[subsection]
\newtheorem{lemma}[theorem]{Lemma}
\newtheorem{corollary}[theorem]{Corollary}
\newtheorem{proposition}[theorem]{Proposition}

\newtheorem*{pvdb}{Conjecture A}

\theoremstyle{definition}
\newtheorem{definition}[theorem]{Definition}
\newtheorem{example}[theorem]{Example}

\theoremstyle{remark}
\newtheorem{remark}[theorem]{Remark}
\theoremstyle{remark}

\newcommand{\bd}{{\bf d}}

\renewcommand{\mod}{\operatorname{mod}}

\newcommand{\und}{\underline}

\newcommand{\OO}{\mathcal{O}}

\newcommand{\DD}{{\mathcal D}}

\newcommand{\KK}{{\mathcal K}}

\newcommand{\BB}{{\mathcal B}}

\newcommand{\hra}{\hookrightarrow}
\newcommand{\lan}{\langle}
\newcommand{\ran}{\rangle}

\newcommand{\GG}{\mathcal{G}}

\newcommand{\ga}{\gamma}

\newcommand{\forg}{\operatorname{forg}}

\newcommand{\ov}{\overline}

\renewcommand{\AA}{{\mathcal A}}

\newcommand{\FF}{{\mathcal F}}

\newcommand{\TT}{\mathcal{T}}
\newcommand{\XX}{{\mathcal X}}
\newcommand{\YY}{{\mathcal Y}}

\renewcommand{\SS}{{\mathcal S}}

\newcommand{\St}{\operatorname{St}}

\renewcommand{\a}{\alpha}
\renewcommand{\b}{\beta}

\newcommand{\la}{\lambda}

\newcommand{\C}{{\mathbb C}}

\newcommand{\Z}{{\mathbb Z}}

\newcommand{\wt}{\widetilde}

\newcommand{\sub}{\subset}

\newcommand{\Per}{\operatorname{Perf}}

\newcommand{\red}{\operatorname{red}}

\renewcommand{\P}{\mathbb{P}}
\newcommand{\A}{\mathbb{A}}

\DeclareMathOperator{\Hom}{\mathrm{Hom}}

\newcommand{\G}{{\mathbb G}}
\newcommand{\om}{\omega}
\newcommand{\si}{\sigma}
\newcommand{\ot}{\otimes}

\begin{document}

\title[Equivariant Derived Categories of Invariant Divisors]{Semiorthogonal
decompositions of equivariant derived categories of invariant divisors}

\author[B. Lim]{Bronson Lim}
\address{BL: Department of Mathematics \\ University of Utah \\ Salt Lake City,
  UT 84102, USA} 
\email{bcl@uoregon.edu}
\author[A. Polishchuk]{Alexander Polishchuk}
\address{AP: University of Oregon and National Research University Higher School of Economics} 
\email{apolish@uoregon.edu}


\begin{abstract}

  Given a smooth variety  \(X\) with an action of a finite group \(G\), and a
  semiorthogonal decomposition of the derived category, \(\mathcal{D}([X/G])\), of
  $G$-equivariant coherent sheaves on \(X\) into subcategories equivalent to
  derived categories of smooth varieties, we construct a similar semiorthogonal
  decomposition for a smooth \(G\)-invariant divisor in \(X\) (under certain
  technical assumptions).  Combining this procedure with the semiorthogonal
  decompositions constructed in \cite{polishchuk-vandenberg-equivariant}, we
  construct semiorthogonal decompositions of some equivariant derived categories
  of smooth projective hypersurfaces.

\end{abstract}

\maketitle

\section{Introduction}
\label{sec:intro}

\subsection{Semiorthogonal decompositions for \(\DD([X/G])\)}
\label{ssec:pvdb-conjecture}

Let \(X\) be a smooth quasiprojective variety over an algebraically closed field \(k\) of
characteristic zero. Suppose \(G\) is a finite group acting on \(X\) by
automorphisms. Then there is a decomposition of the Hochschild homology of the
quotient stack $[X/G]$,
\begin{equation}\label{HH-decomposition-eq}
  HH_\ast([X/G])\cong \bigoplus\limits_{\lambda\in G/{\sim}}
  HH_\ast(X_{\lambda})^{C(\lambda)},
\end{equation}
where \(G/{\sim}\) is the set of conjugacy classes of \(G\),
\(C(\lambda)\) is the centralizer of \(\lambda\), \(X_\lambda\subset X\) is the
invariant subvariety of \(\lambda\), see \cite[Lemma
2.1.1]{polishchuk-vandenberg-equivariant}.  In \cite[Theorem
1.1]{tabuada-vandenberg-additive}, the authors show that the decomposition \eqref{HH-decomposition-eq}
has a motivic origin in an appropriate sense, and that a similar decomposition
exists for any additive invariant of dg-categories.  In
\cite{BGLL-cat-measures-17} a
related decomposition of the equivariant zeta function is given. 


In the case when the geometric quotient $X_\lambda/C(\lambda)$ is smooth one can
identify $HH_\ast(X_\lambda)^{C(\lambda)}$ with $HH_\ast(X_\lambda/C(\lambda))$
(see \cite[Proposition 2.1.2]{polishchuk-vandenberg-equivariant}). Thus, it is natural
to ask whether in some cases the above decomposition can be realized at the
level of derived categories of coherent sheaves. 

\begin{pvdb}[\protect{\cite[Conjecture A]{polishchuk-vandenberg-equivariant}}]
  Assume a finite group \(G\) acts effectively on a smooth variety \(X\), and all the
  geometric quotients \(X_\lambda/C(\lambda)\) are smooth for \(\lambda\in G/{\sim}\). Then there is a
  semiorthogonal decomposition of the derived category \(\DD([X/G])\)
  such that the components \(\mathcal{C}_{[\lambda]}\) of this decomposition are in
  bijection with conjugacy classes in \(G\) and \(\mathcal{C}_{[\lambda]}\cong
  \mathcal{D}(X_\lambda/C(\lambda))\).
	\label{conj:main-conj-a}
\end{pvdb}

This conjecture was verified in \cite{polishchuk-vandenberg-equivariant} in the
case where \(G\) is a complex reflection group of types \(A, B, G_2, F_4\), and
\( G(m,1,n)\) acting on a vector space \(V\), as well as for some actions on
$C^n$, where $C$ is a smooth curve.  Other global results exist for cyclic
quotients, see \cite[Theorem 4.1]{kuz-perry-17} and \cite[Theorem
3.3.2]{L-sum-potentials}, and for quotients of curves, see
\cite[Theorem 1.2]{P-toric-orbifold}. It is shown in \cite[Theorem D]{BGLL-cat-measures-17}
that the above conjecture fails without the assumption that $G$ acts effectively. 
Note that we do not expect a natural bijection in Conjecture A as one can see in simple
examples with the action of a cyclic group (see Example \ref{cyclic-group-action-ex} below). 

Because of the results mentioned above on the analogs of the decomposition \eqref{HH-decomposition-eq}, 
we refer to a semiorthogonal decomposition as in Conjecture A, as {\it motivic semiorthogonal decomposition}.

In all known cases of Conjecture A, the semiorthogonal decompositions are
\(\DD(X/G)\)-linear, where $X/G$ is the geometric quotient,
i.e., the Fourier-Mukai kernels giving the components of the semiorthogonal decomposition
live on the fibered products over $X/G$. We describe this situation in Definition \ref{SOD+def} below.

Let us set for brevity $\ov{X}=X/G$ (we assume that $X$ and $\ov{X}$ are smooth). 
For each conjugacy class $\la$ in $G$ we pick a representative and denote
by $X_\la\sub X$ the corresponding invariant locus. We set $\ov{X}_\la=X_\la/C(\la)$ (the geometric quotient),
\begin{equation}\label{ov-Z-la-eq}
\ov{Z}_\la=\ov{X}_\la\times_{\ov{X}} X.
\end{equation}
Note that $\ov{Z}_\la$ is equipped with a natural $G$-action induced by the $G$-action on $X$, so we have
a diagram
\begin{diagram}
&&[\ov{Z}_\la/G]\\
&\ldTo{q_\la}&&\rdTo{p_\la}\\
\ov{X}_\la&&&&[X/G]
\end{diagram}
in which $q_\la$ is finite flat (since so is the map $X\to X/G$), while $p_\la$ is finite.

For example, for $\la=1$ we have $X_1=X$, $C(1)=G$, $\ov{X}_1=X/G=\ov{X}$, $\ov{Z}_1=X$, $p_1$ is the identity map,
$q_1:[X/G]\to X/G$ is the natural projection. For $\la\neq 1$ the scheme $\ov{Z}_\la$ is typically nonreduced
(see e.g., Example \ref{double-cover-ex}).

\begin{definition}\label{SOD+def}
Let us say that the action of a finite group $G$ on a smooth quasiprojective variety $X$ 
satisfies condition (MSOD)\footnote{MSOD stands for ``motivic semiorthogonal decomposition"} if
\begin{itemize}
\item all the quotients $\ov{X}_\la=X_\la/C(\la)$ are smooth;
\item there exists a collection of objects 
$K_\la$ in $\DD([\ov{Z}_\la/G])$, such that the corresponding Fourier-Mukai functors
$$\Phi_{K_\la}:\DD(\ov{X}_\la)\to \DD([X/G]): F\mapsto p_{\la *}(K_\la\ot q_\la^*F)$$
are fully faithful (here the functors $p_{\la *}$, $\ot$ and $q_\la^*$ are derived); 
\item the corresponding subcategories give a semiorthogonal decomposition
$$\DD([X/G])=\lan \DD(\ov{X}_{\la_1}),\ldots,\DD(\ov{X}_{\la_r})\ran$$
with respect to some total ordering $\la_1,\ldots,\la_r$ on $G/{\sim}$.
\end{itemize}
\end{definition}


\subsection{Restricting (MSOD) to $G$-invariant divisors and application to $S_n$-invariant hypersurfaces}
\label{ssec:globalizing-PVdB-decompositions}

The main observation we make in this paper is that condition (MSOD) is preserved when
passing to sufficiently generic $G$-invariant divisors. Namely, we assume that the action of $G$ on $X$ is effective
and denote by $X^{fr}\sub X$ the open subset on which $G$ acts freely. 
Similarly, for each $\la$ and every connected component $Y\sub X_\la$, let us denote by $W(Y)$
the quotient of $C(\la)$ that acts effectively on $Y$, and let $Y^{fr}\sub X_\la$ denote the open subset on which
$W(Y)$ acts freely. In the case when $X_\la$ is connected we will write $W_\la:=W(X_\la)$.

We will impose the following assumption on a divisor $H$ in $X$:

\medskip

\noindent
($\ast$) for every $\la$ and every connected component
$Y\sub X_\la$, $H$ does not contain $Y$ and $H\cap Y^{fr}$ is dense in $H\cap Y$ 
(in particular, $H\cap X^{fr}$ is dense in $H$).

\medskip

\begin{theorem}\label{SOD+thm}
Assume that the pair $(X,G)$ satisfies (MSOD), and let $H\sub X$ be a smooth $G$-invariant divisor
satisfying ($\ast$). Then the pair $(H,G)$ satisfies (MSOD).
\end{theorem}

We will deduce this result from Kuznetsov's base change for semiorthogonal decompositions
\cite{kuz:sod-base-change}.

To get applications of this theorem, one should start with some pairs $(X,G)$
for which condition (MSOD) is already known. 
We mostly focus on the case of the $S_n$ action on $\mathbb{A}^n$ (in which case 
the semiorthogonal decomposition of the required type was constructed in \cite{polishchuk-vandenberg-equivariant}), and also
consider pairs of the form $(C_1\times\ldots\times C_n,G_1\times\ldots\times G_n)$, where for each $i$, $G_i$ is a finite 
group acting effectively on a smooth curve \(C_i\).

We combine Theorem \ref{SOD+thm} with two simpler procedures: replacing $X$ by
a $G$-invariant open subset and passing to the quotient by a free action of
$\G_m$. This leads us in the case of \([V/S_n]\), where $V=\mathbb{A}^n$, to the following semiorthogonal decomposition for 
the projective hypersurface given by an $S_n$-invariant homogeneous polynomial $f$. 

Note that in this case the conjugacy classes
in $S_n$ are numbered by partitions $\la$ of $n$. For each $\la$, we have the corresponding linear subspace $V_\la$
of invariants and we denote by $W_\la$ the quotient of $C(\la)$ acting effectively on $V_\la$ (see Sec.\ 
\ref{ssec:sym-an-decomp} for details). There is an induced $\G_m$-action on $\ov{V}_\la:=V_\la/W_\la$,
and $f_\la:=f|_{V_\la}$ descends to a quasihomogeneous polynomial $\ov{f}_\la$ on $\ov{V}_\la$.
We denote by 
$X_{\ov{f}_{\la}}\sub \P\ov{V}_\la$ the corresponding weighted projective hypersurface
    stack defined as the quotient of the affine hypersurface $\ov{f}_\la=0$ with the origin removed, by the action of $\G_m$
    (here $\P\ov{V}_\la$ is the weighted projective space stack $[\ov{V}_\la\setminus\{0\}/\G_m]$).  

\begin{theorem}\label{Sn-inv-hyper-thm} 

    Let $f$ be an $S_n$-invariant homogeneous polynomial on $V=\mathbb{A}^n$,
    such that the corresponding projective hypersurface $X_f =\P H(f)\subset
    \mathbb{P}(V)$ is smooth.  Then there exists a semiorthogonal decomposition
    $$\mathcal{D}([X_f/S_n])=\lan \DD(X_{\ov{f}_{\la_1}}),\ldots,\DD(X_{\ov{f}_{\la_r}})\ran,$$ 
    where $\la_1<\ldots<\la_r$ is a total order on partitions of $n$ refining the dominance order.

\end{theorem}

Note that the decomposition of Theorem \ref{Sn-inv-hyper-thm} no longer follows
the pattern of Conjecture A since some components of the decompositions are
themselves derived categories of stacks. The only similarity is that in both
cases there is a birational morphism of stacks inducing a fully faithful
embedding of derived categories via the pull-back (namely, $[X/G]\to X/G$ in Conjecture A and $[X_f/S_n]\to
X_{\ov{f}_{1^n}}$ in Theorem \ref{Sn-inv-hyper-thm}), which is then extended to a
semiorthogonal decomposition of the derived category of the source stack.

\subsection{Outline of paper}
\label{ssec:outline-of-paper}

In Section \ref{sec:sod-prelims}, after some preliminaries, we review Kuznetsov's theory of base change
for semiorthogonal decompositions. In Section
\ref{sec:formalism}, we prove Theorem \ref{SOD+thm} and discuss the procedure of
inducing the semiorthogonal decomposition on the quotient by an
action of a reductive algebraic group. In Section \ref{sec:pvdb-sym}, we
consider applications of Theorem \ref{SOD+thm}. In particular, in
Section \ref{Sn-hypersurface-sec} we prove Theorem \ref{Sn-inv-hyper-thm}. In
Section \ref{other-ex-sec} we consider applications related to the stacks
$[C_1\times\ldots\times C_n/(G_1\times\ldots\times G_n)]$, where $G_i$ is a
finite abelian group acting on a smooth curve $C_i$.

\subsection{Acknowledgments}

The work of the second author is supported in part by the NSF grant DMS-1700642
and by the Russian Academic Excellence Project `5-100'.  He is grateful to Pavel
Etingof for the suggestion to restrict the semiorthogonal decomposition of
\cite{polishchuk-vandenberg-equivariant} to an $S_n$-invariant hypersurface,
which led to this paper.  He also would like to thank Institut de
Math\'ematiques de Jussieu and Institut des Hautes \'Etudes Scientifiques for
hospitality and excellent working conditions.

Both authors are also grateful to the reviewer for drastically improving the
quality of the paper through their extensive comments. 

\subsection{Conventions}
\label{ssec:conventions}

We work over \(\mathbb{C}\). All varieties are assumed to be quasiprojective (in particular, when a finite
group acts on such a variety, the geometric quotient exists).
All stacks are assumed to be quasiprojective DM-stacks in the sense of \cite[Definition 5.5]{kresch:geometry-dm}.
All functors are assumed to be derived. We denote
by \(\mathcal{D}(X)\) (resp., $\Per(X)$), for \(X\) a variety or a stack, the bounded
derived category of coherent sheaves on $X$ (resp., the subcategory of perfect complexes).  When $G$ is an algebraic group
acting on a variety $X$, we denote by $[X/G]$ the corresponding quotient
stack, whereas $X/G$ denotes the geometric quotient (when it exists). 
We always denote by $\P(a_1,\ldots,a_n)$ the weighted projective space {\it stack} obtained as the quotient stack
$[\A^n\setminus\{0\}/\G_m]$, where $\G_m$ acts with the weights $(a_1,\ldots,a_n)$.

\section{Semiorthogonal decompositions and base change for stacks}
\label{sec:sod-prelims}

In this section, we will prove a version of Kuznetsov's base change for semiorthogonal decompositions 
of derived categories of stacks.

\subsection{Semiorthogonal decompositions}
\label{ssec:semiorthogonal-decompositions}

Recall that a \textit{semiorthogonal decomposition} of a triangulated category
\(\mathcal{T}\) is a pair \(\mathcal{A,B}\) of full triangulated subcategories
of \(\mathcal{T}\) such that \(\mathrm{Hom}_{\mathcal{T}}(\mathcal{B,A}) = 0\),
and every object \(t\in \mathcal{T}\) fits into an exact triangle
\[
	b\to t\to a\to b[1]
\]
where \(a\in\mathcal{A}, b\in\mathcal{B}\). In this case, we write
\(\mathcal{T} = \langle \mathcal{A,B}\rangle\). We can iterate this definition to
get semiorthogonal decompositions with any finite number of components
\(\mathcal{A}_1,\ldots,\mathcal{A}_n\) and we write
\[
	\mathcal{T} = \langle \mathcal{A}_1,\ldots,\mathcal{A}_n\rangle.
\]

For an overview of semiorthogonal
decompositions in algebraic geometry, see
\cite{bondal-orlov-sod,bridgeland-derived}.

\subsection{Fourier-Mukai functors}
\label{ssec:fourier-mukai-functorss}

Recall that following \cite{kresch:geometry-dm}, we call a DM-stack $\XX$
{\it quasiprojective} if it has quasiprojective coarse moduli space and 
is a global quotient of a quasiprojective scheme by a reductive algebraic group. 
For example, the quotient stack $[X/G]$, where $G$ is a finite group
acting on a quasiprojective scheme $X$, satisfies these conditions and its coarse moduli space is $X/G$.
By \cite[Prop.\ 5.1]{kresch:geometry-dm}, such a stack has the resolution property, i.e., every coherent sheaf
on it admits a surjective morphism from a vector bundle. Also, such a stack has an affine diagonal.
We denote by $\DD(\XX)$ the bounded derived category of coherent sheaves on $\XX$ and by
$\Per(\XX)\sub\DD(\XX)$ the perfect derived category.

An object
\(\mathcal{K}\in\mathcal{D}(\XX\times \YY)\), whose support is proper over $\YY$,
gives rise to an exact functor \(\Phi_{\mathcal{K}}\colon \Per(\XX)\to
\mathcal{D}(\YY)\) defined by
\[
  \Phi_{\mathcal{K}}(F) =
  \pi_{\YY\ast}(\pi_\XX^\ast F\otimes \mathcal{K}),
\]
where $\pi_\XX:\XX\times \YY\to \XX$ and $\pi_\YY:\XX\times \YY\to \YY$ are the projections.
We will refer to \(\mathcal{K}\) as a \textit{Fourier-Mukai kernel} and
\(\Phi_{\mathcal{K}}\) a \textit{Fourier-Mukai functor}.  
Note that this functor also has a natural extension $\DD_{qc}(\XX)\to \DD_{qc}(\YY)$
to unbounded derived categories of quasicoherent sheaves, which has the right adjoint
\begin{equation}\label{adjoint-fun-formula}
\Phi_\KK^!(G)=\pi_{\XX\ast}\und{\Hom}(\KK,\pi_\YY^!G),
\end{equation}
where $\pi_\YY^!$ is the right adjoint to $\pi_{\YY\ast}$.

The formalism of Fourier-Mukai functors, e.g., as in \cite{huybrechts-fourier},
extends routinely to the case of smooth DM-stacks (see \cite{bn} where Fourier-Mukai functors
are considered in a much more general context).

Note that in the case when $\XX$ and $\YY$ have maps to some DM-stack $\SS$, 
then it is natural to consider relative Fourier-Mukai functors $\Phi_{\KK}$ associated with
kernels $\KK$ on $\XX\times_\SS \YY$, defined in the same way as above. (One gets the same functor
by considering the usual Fourier-Mukai functor associated with the push-forward of $\KK$ with
respect to the morphism $\XX\times_\SS \YY\to \XX\times \YY$.)
We refer to such Fourier-Mukai functors as {\it $\SS$-linear} since they commute with tensoring by the pull-backs
of objects in $\Per(\SS)$ (as one can easily see from the projection formula).
Note that the right adjoint functor $\Phi_\KK^!$ is also $\SS$-linear (see \cite[Lemma 2.34]{kuz:hyperplane-section}).
Also, under appropriate assumptions, such relative Fourier-Mukai functors are compatible with pull-backs
under a base change (see Proposition \ref{FM-base-change-prop} below).

\subsection{Base change for semiorthogonal decompositions}
\label{ssec:base-change-sod}
Here we will recall the result of \cite{kuz:sod-base-change}, on
the base change for semiorthogonal decompositions. For our purposes, we need a slight generalization
to Deligne-Mumford stacks. Throughout, \(\XX,\XX_i,\SS,\TT\)
will be quasiprojective DM stacks in the sense of \cite{kresch:geometry-dm}. 

The following technical definition plays an important role in the base change.

\begin{definition}
    Suppose we have morphisms of quasiprojective DM stacks \(f\colon \XX\to \SS\) and \(\varphi\colon \TT\to \SS\). Then 
    the cartesian diagram 
    \begin{equation}\label{exact-cartesian-square}
      \begin{tikzcd}
        \XX_\TT\ar{d}{f_\TT} \ar{r}{\varphi_\XX} & \XX\ar{d}{f} \\
        \TT \ar{r}{\varphi} & \SS
      \end{tikzcd} 
    \end{equation}
    is called {\bf exact} 
    if the natural map
    \(\varphi^\ast f_\ast\to (f_\TT)_\ast\varphi_\XX^\ast\) is an isomorphism.
In this case we say that the base change \(\varphi\colon \TT\to \SS\) is
{\bf faithful} for the map \(f\). 
\label{def:faithful-base-change}
\end{definition}

For example, the cartesian diagram is exact if either $f$ or $\varphi$ is flat (this is proved similarly
to \cite[Corollary 2.23]{kuz:hyperplane-section}).

\begin{lemma}\label{product-generation-lemma} 
Assume the square \eqref{exact-cartesian-square} is exact cartesian.
Then the category $\Per(\XX_\TT)$ is classically generated by objects of the form 
\(\varphi_\XX^\ast F\otimes f_\TT^\ast G\) with \(F\in
\Per(\XX)\) and \(G\in \Per(\TT)\). 
\end{lemma}

\begin{proof}[1st proof] Arguing as in the case of schemes (see \cite[Lemma 5.2]{kuz:sod-base-change}),
we see that it is enough to check that for every coherent sheaf $F$ on $\XX_\TT$ there exists
a surjection $\varphi_\XX^\ast P_\XX\otimes f_\TT^\ast P_\TT\to F$, where $P_\XX$ (resp., $P_\TT$)
is a vector bundle on $\XX$ (resp., $\TT$). Let us consider the natural map
$$\a:\XX_\TT=\TT\times_\SS \XX \to \TT\times \XX.$$
Since it is obtained by the base change from the diagonal map of $\SS$, $\a$ is affine.
Hence, the adjunction map $\a^*\a_*F\to F$ is surjective (since the image of this map under $\a_*$ is surjective).
Since $\a_*F$ is the union of its coherent subsheaves, we can find a coherent sheaf $G$ on $\TT\times \XX$ 
with a surjective map $\a^*G\to F$.
Now using ample line bundles on the coarse moduli spaces of $\TT$ and $\XX$, as well as vector bundles on $\TT$ and
$\XX$ that have faithful action of the stabilizer subgroups at all geometric points, we can find
vector bundles $P_\XX$ and $P_\TT$ on $\XX$ and $\TT$, respectively, and a surjection 
$P_\XX\ot P_\TT\to G$ (see \cite[Sec.\ 5.2]{kresch:geometry-dm}). Thus, we get the composed surjection
$$\varphi_\XX^*P_\XX\ot f_\TT^*P_\TT\simeq \a^*(P_\XX\ot P_\TT)\to \a^*G\to F.$$

\noindent{\it 2nd proof (sketch)}. Here we use a result of Ben-Zvi, Nadler and Preygel \cite{bn} in the context of derived algebraic geometry.
Using \cite[Prop.\ 2.19]{kuz:hyperplane-section} it is easy to see 
that exactness of a cartesian square is equivalent to its tor-independence. Hence, the derived fiber product
\(\XX_{\TT}^h=\XX\times_{\SS}^L\TT\) is equivalent to the fiber product
\(\XX_{\TT}\). Indeed, let \( (\XX_{\TT}^h,\mathcal{A}^\cdot)\) be the derived
fiber product. Then the cohomology sheaves
\(\mathcal{H}^{-\ast}(\mathcal{A}^\cdot)\) are given by
\(\mathrm{Tor}^{-\ast}_{\mathcal{O}_{\SS}}(\mathcal{O}_{\XX},\mathcal{O}_{\TT})\)
which vanishes for \(\ast\neq 0\).

Now the assertion follows from the equivalence \cite[Theorem 1.2]{bn}:
$$  \mathrm{Perf}(\XX)\otimes_{\mathrm{Perf}(\SS)} \mathrm{Perf}(\TT)
  \xrightarrow{\sim} \mathrm{Perf(}\XX_{\TT}^h)
  \xrightarrow{\sim}\mathrm{Perf}(\XX_{\TT}).$$
\end{proof}

Given an $\SS$-linear Fourier-Mukai functor $\Phi_\KK:\Per(\XX)\to \DD(\YY)$
with the kernel $\KK$ on $\XX\times_\SS\YY$ (with proper support
over $\YY$), one can consider a base change $\varphi:\TT\to \SS$, and the corresponding
$\TT$-linear Fourier-Mukai functor 
$$\Phi_{\KK_\TT}:\Per(\XX_\TT)\to \DD(\YY_\TT)$$
given by the kernel $\KK_\TT$ obtained as the pull-back of $\KK$ with
respect to the natural morphism 
$$\XX_\TT\times_{\TT} \YY_\TT\to \XX\times_{\SS} \YY.$$
The natural question is whether the functors $\Phi_{\KK_\TT}$ and $\Phi_\KK$ are compatible with
the pull-back functors induced by $\varphi$. For our purposes the following criterion
will suffice (see \cite[Lemma 2.42]{kuz:hyperplane-section}).

\begin{proposition}\label{FM-base-change-prop} 
In the above situation assume that the map $\YY\to \SS$ is flat and the base change $\varphi:\TT\to\SS$
is faithful for $\XX\to\SS$.
 In addition, assume that $\XX\to \SS$ is proper and
$\varphi$ has finite Tor-dimension.
Then for $F\in \DD^-(\XX)$ and $G\in \DD^+(\YY)$ (where $\DD^-,\DD^+\sub\DD$ denote bounded
above and bounded below derived categories), one has
$$\Phi_{\KK_\TT}\varphi_\XX^\ast(F)\simeq \varphi_\YY^\ast\Phi_\KK(F),$$
$$\Phi_{\KK_\TT}^!\varphi_{\YY}^{\ast}(G)\simeq\varphi_{\XX}^{\ast}\Phi_\KK^!(G).$$
Here $\varphi_\XX:\XX_\TT\to \XX$ and $\varphi_\YY:\YY_\TT\to \YY$ are the natural projections.
\end{proposition}

\begin{proof} This is proved in the same way as \cite[Lemma 2.42]{kuz:hyperplane-section},
by a calculation on the commutative cube obtained as the product over $\SS$
of the cartesian square
\[
      \begin{tikzcd}
        \XX\times_\SS \YY\ar{d}{} \ar{r}{} & \YY\ar{d}{} \\
        \XX \ar{r}{} & \SS
      \end{tikzcd} 
\]
with the arrow $\varphi:\TT\to \SS$. One has to observe that all the faces of this cube are exact cartesian
and use the base change. The assumption that $\varphi$ is of finite Tor-dimension is used to check that $\varphi^*$ commutes
with $\und{\Hom}$.
\end{proof} 

\begin{remark} The condition that $\YY\to \SS$ is flat in Proposition \ref{FM-base-change-prop}
can be replaced by a weaker condition that 
$\varphi:\TT\to\SS$ is faithful for $\YY\to \SS$ and for $\XX\times_\SS \YY\to \SS$, as is
done in  \cite[Lemma 2.42]{kuz:hyperplane-section}.
Note that there is a slight mistake in the proof of \cite[Theorem 6.4]{kuz:sod-base-change} where the above
faithfulness assumption (as well as the assumption of smoothness of $\SS$) is omitted.
\end{remark}

The following result is similar to (but more special than) \cite[Thm.\ 6.4]{kuz:sod-base-change}.

\begin{lemma}\label{full-faithful-lem} 
Assume that $\XX$, $\YY$ and $\SS$ are smooth, $\SS$ is separated, the morphisms $\XX\to \SS$ and $\YY\to \SS$ are proper, and
the morphism $\YY\to \SS$ is flat.
Let $\Phi_\KK:\DD(\XX)\to \DD(\YY)$ be the $\SS$-linear Fourier-Mukai functor associated with
a kernel $\KK$ in $\DD(\XX\times_\SS \YY)$,
and let $\varphi:\TT\to\SS$ be a faithful base change for both $\XX$ and $\YY$. Assume that 
the support of $\KK$ is proper over $\XX$ and that $\Phi_\KK$ is
fully faithful. Then $\Phi_{\KK_\TT}:\Per(\XX_\TT)\to \DD(\YY_\TT)$ is also fully faithful.
\end{lemma}

\begin{proof} First, we claim that the functor $\Phi_{\KK}^!$ sends $\DD(\YY)$ to $\DD(\XX)$.
To this end we observe that $\Phi_\KK$ can be computed as the absolute Fourier-Mukai functor $\Phi_{\KK'}$,
where the kernel $\KK'$ is given by the push-forward of $\KK$ with respect to the finite
morphism $\XX\times_\SS \YY\to \XX\times \YY$ (finiteness of this morphism follows from the
finiteness of the diagonal morphism for $\SS$). Since $\XX\times \YY$ is smooth, the right adjoint functor
$\Phi^!_{\KK'}$ sends $\DD(\YY)$ to $\DD(\XX)$ (as follows from formula \eqref{adjoint-fun-formula}).

Thus, the fact that $\Phi_\KK$ is fully faithful on $\DD(\XX)$ implies that the natural morphism
\begin{equation}\label{Phi-K-adj-can-morphism}
F\to \Phi_\KK^!\Phi_\KK(F)
\end{equation}
is an isomorphism for $F\in\DD(\XX)$.

Now to check that $\Phi_{\KK_\TT}$ is fully faithful on $\Per(\XX_\TT)$,
by Lemma \ref{product-generation-lemma}, it is enough to check that the morphism
$$\wt{F}\to \Phi_{\KK_\TT}^!\Phi_{\KK_\TT}(\wt{F})$$ 
is an isomorphism for objects of the form
\(\wt{F}=\varphi_\XX^\ast F\otimes f_\TT^\ast G\), where \(F\in
\Per(\XX)\) and \(G\in \Per(\TT)\).
But this easily follows from $\TT$-linearity of our functors and from Proposition \ref{FM-base-change-prop} (note that
$\varphi$ has finite Tor-dimension since $\SS$ is smooth):
$$\Phi_{\KK_\TT}^!\Phi_{\KK_\TT}(\varphi_\XX^\ast F\otimes f_\TT^\ast G)\simeq
\varphi_\XX^\ast\Phi_{\KK}^!\Phi_{\KK}(F)\otimes f_\TT^\ast G,$$
so the needed assertion follows from the fact that \eqref{Phi-K-adj-can-morphism} is an isomorphism.
\end{proof}

Suppose we have an \(\SS\)-linear semiorthogonal decomposition 
\[
\Per(\XX) = \langle \AA_1,\ldots,\AA_m\rangle.
\]
Let us define subcategories \(\AA_{i\TT}\subset \Per(\XX_\TT)\) by the formula
\[
\AA_{iT} = \langle \varphi_\XX^\ast A\otimes f_\TT^\ast G\rangle_{A\in\AA_i, G\in \Per(\TT)}.
\]
The following two theorems are analogs of \cite[Thm.\ 5.6]{kuz:sod-base-change} and \cite[Thm.\ 6.4]{kuz:sod-base-change}.

\begin{theorem}\label{general-bc-thm}
    Suppose \(\varphi\colon \TT\to \SS\) is faithful
    for \(f\colon \XX\to \SS\). Assume that there is an
    \(\SS\)-linear semiorthogonal decomposition 
    \[
      \Per(\XX) = \langle \AA_1,\ldots,\AA_m\rangle.
    \]
    Then the subcategories $\AA_{iT}$ form a \(\TT\)-linear semiorthogonal decomposition of \(\Per(\XX_{\TT})\),
    \[
        \Per(\XX_{\TT}) = \langle \AA_{1\TT},\ldots,\AA_{m\TT}\rangle.
    \]
\end{theorem}

\begin{proof}
As in \cite{kuz:sod-base-change}, the semiorthogonality \(\langle
\AA_{i\TT},\AA_{j\TT}\rangle\) for \(i>j\) follows from faithful base change. 
Now Lemma \ref{product-generation-lemma} implies that the subcategories
$\AA_{1\TT},\ldots,\AA_{m\TT}$ generate $\Per(\XX_{\TT})$, and the assertion follows.
\end{proof}

\begin{theorem}
Suppose $\XX$ and $\SS$ are smooth, $\SS$ is separated, the morphism \(f\colon \XX\to \SS\) is flat and proper, and there is an
  \(\SS\)-linear semiorthogonal decomposition
  \[
    \DD(\XX) = \langle \DD(\XX_1),\ldots,\DD(\XX_m)\rangle,
  \]
  where for \(i=1,\ldots,m\), the stacks
  $\XX_i$ are smooth, the maps \(\XX_i\rTo{f_i}\SS\) are proper, and
  the embedding functors
  $\Phi_i:\DD(\XX_i)\to \DD(\XX)$
  are given by some kernels $\KK_i$ in $\DD(\XX_i\times_{\SS} \XX)$.
  Assume now that 
  \(\varphi\colon \TT\to \SS\) is a base change, faithful for $f$ and for each $f_i$.
   Set \(\XX_{i\TT} =\XX_i\times_\SS\TT\). Then the pullbacks \(\KK_{i\TT}\) of $\KK_i$ to
  \(\XX_{i\TT}\times_\TT \XX_\TT\) define fully faithful functors
  \[
    \Phi_{i\TT}\colon \Per(\XX_{i\TT})\to \Per(\XX_\TT).
  \]
and their images give a \(\TT\)-linear semiorthogonal decompositon
  \[
    \Per(\XX_{\TT}) = \langle \Per(\XX_{1\TT}),\cdots,\Per(\XX_{m\TT})\rangle.
  \]
  \label{thm:kernel-compatibility}
\end{theorem}

\begin{proof}
Let us set \(\AA_i=\Phi_i\Per(\XX_i)\).   By Theorem \ref{general-bc-thm}, we get
a semiorthogonal decomposition of $\Per(\XX_\TT)$ into the subcategories $\AA_{i\TT}$
(note that $\DD(\XX)=\Per(\XX)$ and $\DD(\XX_i)=\Per(\XX_i)$ by smoothness).

Now we observe that for $F\in \Per(\XX_i)$ and $G\in \Per(\TT)$, we have an isomorphism
$$\varphi_{\XX}^\ast(\Phi_i F)\ot f_\TT^\ast(G)\simeq \Phi_{i\TT}(\varphi_{\XX_i}^\ast F)\ot f_\TT^\ast(G)\simeq
\Phi_{i\TT}(\varphi_{\XX_i}^\ast F\ot f_{i\TT}^\ast(G)),$$
where we used commutation of relative Fourier-Mukai functors with the pull-back 
(see Proposition \ref{FM-base-change-prop})
and $\TT$-linearity of
$\Phi_{i\TT}$ (and $(f_{i\TT}, \varphi_{\XX_i})$ have the same meaning for $\XX_i$ as $(f,\varphi)$ for $\XX$).
Using Lemma \ref{product-generation-lemma} for $\XX$ and $\XX_i$, we deduce that
the image of \(\Per(X_{i\TT})\) under  \(\Phi_{i\TT}\) is exactly \(\AA_{i\TT}\sub\Per(\XX_\TT)\).
Finally, by Lemma \ref{full-faithful-lem}, the functors  \(\Phi_{i\TT}\) are fully faithful.   
\end{proof}

An easy example of the faithful base change is restricting to an open subset. In particular, we deduce
that condition (MSOD) is preserved when passing to a $G$-invariant open subset.

\begin{corollary}\label{open-res-cor} 
Assume that the pair $(X,G)$ satisfies (MSOD), and let $U\sub X$ be a $G$-invariant open subset.
Then the pair $(U,G)$ also satisfies (MSOD), and the corresponding kernels on 
$\ov{U}_\la\times_{\ov{U}} U$ are obtained as pull-backs of the kernels on $\ov{X}_\la\times_{\ov{X}}X$.
\end{corollary}

\begin{proof}
To deduce this from Theorem \ref{thm:kernel-compatibility}, we observe that $U$ is the preimage of the open subset
$\ov{U}=U/G\sub \ov{X}$, and $\ov{U}_\la\sub \ov{X}_\la$ is the preimage of $\ov{U}$ under the map 
$\ov{X}_\la\to \ov{X}$. Note that the natural map of stacks over $[X/G]$,
$$[U/G]\to [X/G]\times_X U\simeq [X/G]\times_{\ov{X}}\ov{U},$$
is an equivalence, since it becomes an isomorphism after the base change $X\to [X/G]$.

Thus, we can apply Theorem \ref{thm:kernel-compatibility} to the faithful base change $\ov{U}\to \ov{X}$.
Furthermore, we have
$$(\ov{X}_\la\times_{\ov{X}} X)\times_{\ov{X}}\ov{U}\simeq\ov{X}_\la\times_{\ov{X}} U\simeq
\ov{X}_\la\times_{\ov{X}}\ov{U}\times_{\ov{U}} U\simeq\ov{U}_\la\times_{\ov{U}} U,$$
so the new kernels live on the correct spaces. 
\end{proof}

\subsection{Products}

We observe that condition (MSOD) is compatible with products 
(see \cite[Corollary 5.10]{kuz:sod-base-change} for a more general result).

\begin{lemma}\label{product-lem} 
Let $G$ (resp., $G'$) be a finite group acting on a smooth variety $X$ (resp., $X'$),
and assume that condition (MSOD) is satisfied for $(X,G)$ (resp., $(X',G')$).
Then condition (MSOD) is also satisfied for the action of
$G\times G'$ on $X\times X'$.
\end{lemma}

\begin{proof} Let $(\la,\la')$ be a conjugacy class in $G\times G'$. The corresponding scheme
$$\ov{Z}_{\la,\la'}=(\ov{X}_\la \times \ov{X'}_{\la'})\times_{\ov{X}\times\ov{X'}} (X\times X')$$
is naturally identified with $\ov{Z}_\la\times \ov{Z}_{\la'}$ so we can define the kernel
$\KK_{\la,\la'}$ on $\ov{Z}_{\la,\la'}$ as the exterior tensor product
$\KK_\la\boxtimes \KK_{\la'}$. 
The corresponding functor sends $F\boxtimes F'$, where $F\in\DD(\ov{X}_\la)$, $F'\in\DD(\ov{X'}_\la)$, 
to $\Phi_{\KK_\la}(F)\boxtimes\Phi_{\KK_{\la'}}(F')$. Thus, computing morphisms between such objects 
using the K\"unneth formula, we deduce that the functor $\Phi_{\KK_\la\boxtimes \KK_{\la'}}$ is fully faithful on objects of the form $F\boxtimes F'$,
and hence on all objects. Similarly, we check semiorthogonality
between the images and generation. Thus, we get a semiorthogonal decomposition of $\DD([X\times X'/(G\times G')])$
with respect to any total ordering of conjugacy classes in $G\times G'$ compatible with the partial
order $(\la_1,\la'_1)\le (\la_2,\la'_2)$ if $\la_1\le \la_2$ and $\la'_1\le \la'_2$ (where we use the total orders on
conjugacy classes in $G$ and $G'$ corresponding to the original decompositions).
\end{proof}

\section{$G$-invariant divisors and proof of Theorem \ref{SOD+thm}}
\label{sec:formalism}

\subsection{Smooth $G$-invariant divisors}
\label{ssec:smoothness-quotients}

Throughout this section we fix a smooth connected variety $X$ with an effective action of a finite group $G$, such that
$\ov{X}=X/G$ is smooth. We denote by $X^{fr}\sub X$ the open subset on which the action of $G$ is free.
Recall that $X_\la\sub X$ denotes the $\la$-invariant locus in $X$ (where $\la$ runs over a set of representatives
of conjugacy classes in $G$), and $\ov{X}_\la=X_\la/C(\la)$. Note that the ideal sheaf of $X_\la$ is generated locally
by elements of the form $\la^*(f)-f$, with $f\in\OO_X$, and with this subscheme structure $X_\la$ is smooth.

For each $\la$
and every connected component $Y\sub X_\la$, let us denote by $W(Y)$
the quotient of $C(\la)$ that acts effectively on $Y$, and let $Y^{fr}\sub X_\la$ denote the open subset on which
$W(Y)$ acts freely.

\begin{lemma}\label{flat-quotient-map-lem} 
The morphism $X\to \ov{X}=X/G$ is finite flat of degree $|G|$.
\end{lemma}

\begin{proof}
It is well known that the morphism $X\to X/G$ is finite surjective of degree $|G|$ (see \cite[Ch.\ II.7]{Mum-ab-var}).
Since $X$ and $\ov{X}$ are smooth, it is flat by the miracle flatness theorem.
\end{proof}

Now let $H\subset X$ be a smooth $G$-invariant divisor. 
We have the induced action of $G$ on $H$, so we can consider varieties $\ov{H}=H/G$,
$H_\la\sub H$ and $\ov{H}_\la=H_\la/C(\la)$. 
It is easy to see that 
$$H_\la=H\cap X_\la,$$
the scheme-theoretic intersection.

We start by observing that the smoothness of the geometric quotient is preserved
upon passing to a smooth $G$-invariant divisor.

\begin{proposition} For a smooth $G$-invariant divisor $H\sub X$, the quotient $\ov{H}=H/G$ is smooth.
  \label{prop:geom-quotient-divisor-smooth}
\end{proposition}

\begin{proof}
  Assume first that \(x\in H\) is a \(G\)-invariant point. We can linearize the action in a formal neighborhood of \(x\) in \(X\), so the
  divisor \(H\) will be a $G$-invariant hyperplane. Since \(X/G\) is smooth at \(x\),
  \(G\) is generated by pseudo-reflections. Hence, the same is true for the induced action of $G$ on $H$, so the quotient \(H/G\) is
  smooth at \(x\).

  Now let \(x\in H\) be arbitrary, and let $\mathrm{St}_x\sub G$ denote the stabilizer subgroup of $x$. 
  By Luna's \'etale slice theorem,
  \cite{luna-slices-etale}, the map \(X/\mathrm{St}_x\to X/G\) is \'etale near the image of
  \(x\) in $X/\mathrm{St}_x$. Since \(X/G\) is smooth, this implies \(X/\mathrm{St}_x\)
  is also smooth at $x$. Thus, \(H/\mathrm{St}_x\) is smooth at $x$, by the previous
  argument. Since the mapping \(H/\mathrm{St}_x\to H/G\) is \'etale at the image of $x$ in  \(H/\mathrm{St}_x\), 
  we conclude that \(H/G\) is smooth at the image of \(x\) (using \cite[Theorem 17.11.1]{ega4-4}).
\end{proof}

\begin{corollary}\label{H-la-smooth-cor} 
For any $\la$, if $\ov{X}_\la$ is smooth then $\ov{H}_\la$ is smooth.
\end{corollary}

\begin{proof}
	The scheme \(X_\lambda\) is smooth as the fixed locus of a finite order
	automorphism in $X$. Similarly, \(H_\lambda\) is smooth as \(H\) is smooth and
	\(H_\lambda\) is the fixed locus of \(\lambda\). Thus, $H_\la=H\cap X_\la$
	is a smooth divisor in $X_\la$, so we can apply 	Proposition \ref{prop:geom-quotient-divisor-smooth}.
\end{proof}

\begin{lemma}\label{cartesian-divisor-lem} Assume that $H\cap X^{fr}$ is dense in $H$.
Then the square
\begin{diagram}
H&\rTo{}& \ov{H}\\
\dTo{}&&\dTo{}\\
X&\rTo{}&\ov{X}
\end{diagram}
is exact cartesian.
\end{lemma}

\begin{proof}
Since $\ov{H}$ (resp., $H$) is a divisor in $\ov{X}$ (resp., $X$), and both $X$ and $\ov{X}$ are smooth, by \cite[Cor.\ 2.27]{kuz:hyperplane-section},
it is enough to check that our square is cartesian.
By Lemma \ref{flat-quotient-map-lem},
both maps $X\to \ov{X}$ and $H\to\ov{H}$ are finite flat of degree $|G|$ (here we use the assumption that
$H\cap X^{fr}$ is dense in $H$ and Proposition \ref{prop:geom-quotient-divisor-smooth} which assures that
$\ov{H}$ is smooth). The embedding of $H$ into $X$
factors through $X\times_{\ov{X}}\ov{H}$ which is a closed subscheme of $X$. 
Thus $H\sub X\times_{\ov{X}}\ov{H}$ is a closed embedding of schemes, both of which are 
finite flat of degree $|G|$ over $\ov{H}$, and the assertion follows.
\end{proof}

\begin{proposition}\label{two-cartesian-sq-prop} 
Assume that for some $\la$, $\ov{X}_\la$ is smooth, and that $H$ satisfies condition $(\ast)$ introduced before
Theorem \ref{SOD+thm}.
Then both squares in the diagram
\begin{equation}
\begin{diagram}\label{H-X-la-cart-diagram}
H_\la&\rTo{}& \ov{H}_\la&\rTo{}&\ov{H}\\
\dTo{}&&\dTo{}&&\dTo{}\\
X_\la&\rTo{}&\ov{X}_\la&\rTo{}&\ov{X}
\end{diagram}
\end{equation}
are exact cartesian.
\end{proposition}

\begin{proof}
By \cite[Cor.\ 2.27]{kuz:hyperplane-section}, it is enough to check that these squares are cartesian.
First, we observe that by
Lemma \ref{cartesian-divisor-lem}, the right square in the diagram
\begin{diagram}\label{H-X-la-cart-diagram}
H_\la&\rTo{}& H &\rTo{}&\ov{H}\\
\dTo{}&&\dTo{}&&\dTo{}\\
X_\la&\rTo{}& X &\rTo{}&\ov{X}
\end{diagram}
is cartesian. Since $H_\la$ is the scheme-theoretic intersection of $X_\la$ with $H$, the left square in this diagram is also cartesian.
Hence, the big rectangle in this diagram, which is the same
as the big rectangle in diagram \eqref{H-X-la-cart-diagram}, is cartesian.

Next, applying Lemma \ref{cartesian-divisor-lem} to every connected component $Y$ of $X_\la$, the group $W(Y)$ acting on it
and the divisor $Y\cap H_\la$, we
get that the left square is cartesian. Since the map $X_\la\to \ov{X}_\la$ is flat and surjective,
and the big rectangle is cartesian, we derive the same for the right square
(by checking that the map $\ov{H}_\la\to \ov{X}_\la\times_{\ov{X}}\ov{H}$ of $\ov{X}_\la$-schemes
becomes an isomorphism after the base change $X_\la\to \ov{X}_\la$).
\end{proof}

\subsection{Proof of Theorem \ref{SOD+thm}}

The condition (MSOD) for $(X,G)$ gives an $\ov{X}$-linear semiorthogonal decomposition of $\DD([X/G])$. To deduce from this
the same condition for $(H,G)$,
we want to apply Theorem \ref{thm:kernel-compatibility} to the base change of the morphism $[X/G]\to \ov{X}$ with respect to
the morphism $\ov{H}\to \ov{X}$, i.e., to $\XX=[X/G]$, $\SS=\ov{X}$, $\XX_i=\ov{X}_\la$ and $\TT=\ov{H}$.

Note that by assumption, $[X/G]$, $\ov{X}$ and $\ov{X}_\la$ are smooth. 
Also, the morphisms $[X/G]\to X/G=\ov{X}$ and $\ov{X}_\la\to \ov{X}$ are proper, so Theorem \ref{thm:kernel-compatibility} is applicable.

We claim that the corresponding diagram
\begin{diagram}
[H/G]&\rTo{}& \ov{H}\\
\dTo{}&&\dTo{}\\
[X/G]&\rTo{}&\ov{X}
\end{diagram}
is exact cartesian. Indeed, by Lemma \ref{cartesian-divisor-lem} the natural $1$-morphism of stacks 
$[H/G]\to [X/G]\times_{\ov{X}}\ov{H}$ over $[X/G]$ 
becomes an isomorphism $H\to X\times_{\ov{X}}\ov{H}$ after the base change $X\to [X/G]$. Hence, it is an equivalence.

Also, by Proposition \ref{two-cartesian-sq-prop}, the base change of
$\ov{X}_\la$ gives us $\ov{H}_\la$. Note that by Corollary \ref{H-la-smooth-cor}, $[H/G]$ and $\ov{H}_\la$ are smooth.
Thus, Theorem \ref{thm:kernel-compatibility} gives an $\ov{H}$-linear semiorthogonal
decomposition of $\Per([H/G])=\DD([H/G])$ with the components $\Per(\ov{H}_\la)=\DD(\ov{H}_\la)$.  
Furthermore, the kernel on $\ov{H}_\la\times_{\ov{H}} H$
giving the functor $\DD(\ov{H}_\la)\to \DD([H/G])$ is given by the pullback of $K_\la$.
Thus, all conditions of Definition \ref{SOD+def} are satisfied for the action of $G$ on $H$.
\qed

\subsection{Subschemes $Z_\la$ in $X_\la\times X$}\label{Z-la-sec}

The content of this subsection is not used anywhere else in the paper. Here we discuss certain natural subschemes of $X_\la\times_{\ov{X}} X$.
In some situations considered in \cite{polishchuk-vandenberg-equivariant} they are related to the equivariant cohomology of the Springer fibers.

Assume that $G$ is a finite group acting on a quasiprojective smooth variety $X$, such that all $\ov{X}_\la=X_\la/C(\la)$
are smooth.

For each $\la\in G/{\sim}$, let us define the closed subscheme $Z_\la\sub X_\la\times_{\ov{X}} X$ by
$$Z_\la:=X_\la\times_{\ov{X}} X,$$
and let $Z_{\la}^{\red}\sub Z_\la$ be the corresponding reduced subscheme. Note that
$Z_\la^{\red}$ 
is the union of the graphs of the embeddings $g:X_\la\to X$, for $g$ running over $G/C(\la)$.
\footnote{Our notation is different from \cite{polishchuk-vandenberg-equivariant} where $Z_\la$ denotes
the reduced subscheme.}
It is easy to see that $\ov{Z}_\la=Z_\la/C(\la)$, where $\ov{Z}_\la$ is defined by \eqref{ov-Z-la-eq}.
Also, we have $\ov{Z}_\la^{\red}=Z_\la^{\red}/C(\la)$ since both are subschemes of $\ov{Z}_\la$
defined by a nilpotent ideal and since the quotient of a reduced scheme by a finite group is reduced.

The schemes $\ov{Z}_\la^{\red}$ play an important role in the work \cite{polishchuk-vandenberg-equivariant}:
in the examples considered in that paper (see also Sec.\ \ref{ssec:sym-an-decomp} below),
the kernels of the functors defining the semiorthogonal decompositions of $\DD([X/G])$ are given
by some vector bundles on $\ov{Z}_\la^{\red}$.

The simplest example below shows that $Z_\la$ and $\ov{Z}_\la$ are typically nonreduced.

\begin{example}\label{double-cover-ex}
Let $X=\A^1$, $G=\Z_2$ acting on $\A^1$ by $x\mapsto -x$. We can take $t=x^2$ as a coordinate
on $X/G\simeq \A^1$. Then for $\la\neq 1$, $\ov{Z}_\la=Z_{\la}\sub\A^1$ is the subscheme corresponding to the ideal $(x^2)$.
\end{example}

Note that in the special case $\la=1$, we have $\ov{Z}_1=X$, which is reduced.
It turns out that the subscheme $Z_1=X\times_{\ov{X}}X \sub X\times X$ is still reduced (and is equal to the union
of the graphs of all $g\in G$ acting on $X$) provided the action of $G$ is effective. 

\begin{lemma}\label{Z1-red-lem} Assume that the action of $G$ on $X$ is effective, and the schemes $X$ and $X/G$ are
smooth. Then $Z_1$ is reduced.
\end{lemma}

\begin{proof} Since the projection $X\to \ov{X}$ is finite flat, the same is true about the projection $p_1:Z_1\to X$.
Thus, $p_{1*}\OO_{Z_1}$ is locally free over $\OO_X$, in particular, it is torsion free as an $\OO_X$-module.
Furthermore, the fact that the action of $G$ is effective implies that $Z_1$ is reduced over a generic point of $X$.
Hence, the nilradical of $\OO_{Z_1}$ would give a torsion submodule $p_{1*}\OO_{Z_1}$, so this nilradical has to be trivial.
\end{proof}

\subsection{Passing to the quotient stacks}

Note that a general theory of inducing semiorthogonal decompositions on quotients of varieties by actions of reductive groups
is considered in \cite{elagin:descent-sod}. We need an analogous result where instead of varieties we
consider stacks of the form $[X/G]$.

Namely, assume that condition (MSOD) holds for a pair $(X,G)$.
Assume in addition that there is
a reductive algebraic group $\G$ acting on $X$, such that the actions of $G$
and $\G$ commute and $[X/\G]$ is a DM-stack. In particular, the subvarieties $X_\la$ acquire the action of $C(\la)\times \G$
and there is an induced action of $\G$ on $\ov{X}_\la=X_\la/C(\la)$ and on $\ov{Z}_\la$. This
action is compatible with the projections to $\ov{X}_\la$ and to $X$.
Assume also that each kernel $\KK_\la$ in $\DD([\ov{Z}_\la/G])$
comes from an object $\wt{\KK}_\la$ in $\DD([\ov{Z}_\la/(G\times \G)])$.
In this case each $\wt{\KK}_\la$ defines the Fourier-Mukai functor
$$\Phi^{\G}_{\KK_\la}:\DD([\ov{X}_\la/\G])\to \DD([X/(G\times\G)])$$
that fits into a commutative square
\begin{diagram}
\DD([\ov{X}_\la/\G]) &\rTo{\Phi_{\KK_\la}^\G}& \DD([X/(G\times\G)])\\
\dTo{}&&\dTo{}\\
\DD(\ov{X}_\la)&\rTo{\Phi_{\KK_\la}}& \DD([X/G])
\end{diagram}
where the vertical arrows are given by forgetting the $\G$-action.
In other words, $\Phi^{\G}_{\KK_\la}$ is defined by the same formula as $\Phi_{\KK_\la}$, but we view
the result as an object of the $G\times\G$-equivariant derived category.

\begin{lemma}\label{passing-to-quot-lem}
The functors $\Phi^\G_{\KK_\la}$ are fully faithful and their images give a
semiorthogonal decomposition
$$\DD([X/(G\times\G)])=\lan \DD([\ov{X}_{\la_1}/\G]),\ldots,\DD([\ov{X}_{\la_r}/\G])\ran.$$
\end{lemma}

\begin{proof}
For a pair of objects $\FF,\GG\in \DD([\ov{X}_\la/\G])$, we have a commutative square
\begin{diagram}
\Hom_{\DD([\ov{X}_\la/\G])}(\FF,\GG) &\rTo{\Phi^\G_{\KK_\la}}& \Hom_{\DD([X/(G\times\G)])}(\Phi^\G_{\KK_\la}(\FF),\Phi^\G_{\KK_\la}(\GG))\\
\dTo{\forg}&&\dTo{\forg}\\
\Hom_{\DD(\ov{X}_\la)}(\FF,\GG)^{\G}&\rTo{\Phi_{\KK_\la}}& \Hom_{\DD([X/G])}(\Phi_{\KK_\la}(\FF),\Phi_{\KK_\la}(\GG))^\G
\end{diagram}
in which the vertical arrows are isomorphisms since $\G$ is reductive. 
Furthermore, since $\Phi_{\KK_\la}$ is fully faithful, the bottom horizontal arrow is
an isomorphism. Hence, the top horizontal arrow is also an isomorphism, i.e., $\Phi^\G_{\KK_\la}$ is fully faithful.

Similarly, if $\Hom(\Phi_{\KK_\la}(\cdot),\Phi_{\KK_\mu}(\cdot))=0$ then by passing to $\G$-invariants, we deduce that
$\Hom(\Phi_{\KK_\la}^\G(\cdot),\Phi_{\KK_\mu}^\G(\cdot))=0$. Hence, the semiorthogonality still holds for the images of
$\Phi_{\KK_\la}^\G$.

Finally, to see that the images of \( \Phi_{\KK_\la}^\G\) generate $\DD([X/(G\times\G)])$, we
observe that the right adjoint functors $(\Phi^\G_{\KK_\la})^!$ and $\Phi_{\KK_\la}^!$ to $\Phi^\G_{\KK_\la}$ and
$\Phi_{\KK_\la}$ are still compatible with 
the forgetful functors, i.e., we have a commutative diagram
\[
\begin{tikzcd}
  \DD([X/G\times\G])\ar{r}{(\Phi_{\KK_\lambda}^\G)^!}\ar{d}{\forg} & \DD([\ov{X}_\lambda/\G])
  \ar{d}{\forg} \\
  \DD([X/G])\ar{r}{\Phi_{\KK_\lambda}^!} & \DD(\ov{X}_\lambda)
\end{tikzcd}
\]
Indeed, for any $F\in \DD([X/G\times\G])$, we can define $(\Phi_{\KK_\la}^\G)^!(F)$ by
the same formula for $\Phi_{\KK_\la}^!$ (see \eqref{adjoint-fun-formula}) understood in terms of
equivariant categories. Then the above commutative diagram holds and the adjunction follows 
from the chain of isomorphisms
\begin{align*}
&\Hom_{\DD([X/G\times\G])}(\Phi_{\KK_\la}^\G(F),F')\simeq \Hom_{\DD([X/G])}(\Phi_{\KK_\la}(F),F')^{\G}\\
&\simeq
\Hom_{\DD(\ov{X}_\la)}(F,\Phi_{\KK_\la}^!(F'))^\G
\simeq \Hom_{\DD([\ov{X}_\lambda/\G])}(F,(\Phi_{\KK_\la}^\G)^!(F')).
\end{align*}

Now suppose \(\FF\in\DD([X/G\times\G])\) is right
orthogonal to the images of \( \Phi_{\KK_\lambda}^\G\). Then by adjointness,
\((\Phi_{\KK_\lambda}^\G)^!(\FF)=0\)
for all \(\lambda\). Thus, using the above commutative diagram, we obtain that $\forg(\FF)$
is right orthogonal to the images of all $\Phi_{\KK_\la}$. Using the original semiorthogonal decomposition,
we conclude that \(\forg(\FF)=0\). But the forgetful functor is conservative, so
\(\FF=0\).
\end{proof}

\begin{remark} The above lemma can also be deduced from the conservative descent theorem of Bergh and Schn\"urer \cite{Bergh-Schn}.
\end{remark}

\section{Examples of semiorthogonal decompositions obtained from Theorem \ref{SOD+thm}}
\label{sec:pvdb-sym}

\subsection{Motivic decomposition for \(\mathcal{D}([\mathbb{A}^n/S_n])\)}
\label{ssec:sym-an-decomp}

Now we will focus on the case of the standard action of the symmetric group
$S_n$ on the affine \(n\)-space, \(V=\mathbb{A}^n\). In this case $\ov{V}=V/S_n$ is still the affine space $\mathbb{A}^n$
and the morphism $V\to \ov{V}$ is given by the elementary symmetric polynomials.

The conjugacy classes of \(S_n\) are labeled by partitions
\(\lambda=(\la_1\ge \la_2\ge\ldots)\) of \(n\). Recall that the \textit{dominance partial ordering} \(\leq\) on
partitions of \(n\) is defined by \(\lambda\geq \mu\) if
\(\lambda_1+\cdots+\lambda_i\geq \mu_1 + \cdots + \mu_i\) for all \(i\geq 1\).
Note that \( (n)\) is the biggest partition and \( (1^n)\) is the smallest.

For each partition \(\lambda=(\la_1\ge \la_2\ge\ldots)\) of \(n\), we choose as a representative of the corresponding conjugacy
class the permutation $(1\ldots\la_1)(\la_1+1\ldots \la_1+\la_2)\ldots$ The corresponding fixed locus
\(V_\lambda\sub V\) is isomorphic to $\mathbb{A}^{\ell(\la)}$, where $\ell(\la)$ is the length of $\la$, and the isomorphism
is given by
\begin{equation}\label{V-la-partition-eq}
(y_1,\ldots,y_\ell)\mapsto (\underbrace{y_1,\ldots,y_1}_{\la_1},\underbrace{y_2,\ldots,y_2}_{\la_2},\ldots,
\underbrace{y_\ell,\ldots,y_\ell}_{\la_\ell}).
\end{equation}

If we write $\la$ exponentially: $\la=(1^{r_1},2^{r_2},\ldots,p^{r_p})$ (where $r_i$ is the multiplicity of the part $i$),
then we have 
$$C(\la)\simeq\prod_{i=1}^p ((\Z/i\Z)^{r_i}\rtimes S_{r_i}),$$
and the quotient of $C(\la)$ that acts effectively on $V_\lambda$ is the
group \(W_\lambda \simeq \prod_i S_{r_i}\).
We have
$$\ov{V}_\lambda=V_\lambda/W_\la\simeq \A^{\ell(\la)}=\prod_{i=1}^p \A^{r_i},$$ 
and the quotient map $V_\la\to \ov{V}_\la$ is the product of maps $\A^{r_i}\to \A^{r_i}$
given by the elementary symmetric functions.
Under the identiication \eqref{V-la-partition-eq},
the open subset $V_\lambda^{fr}\sub V_\lambda$, on which $W_\la$ acts freely, is given by
\begin{equation}\label{V-la-fr-eq}
V_\la^{fr}=\{(y_1,\ldots,y_\ell) \ | \ y_i\neq y_j  \text{  whenever } \la_i=\la_j, i\neq j\}.
\end{equation}

Recall that we have the reduced subscheme \(Z^{\red}_\lambda\subset V_\lambda\times V\),
invariant under the action of $W_\la\times S_n$ (see Sec.\ \ref{Z-la-sec}). 
Explicitly, this is the union of graphs of all maps $V_\lambda\to V: x\mapsto \si x$ over $\si\in S_n/C(\la)$.
The quotient $Z^{\red}_\la/W_\la$ can be identified with the reduced subscheme 
$$\ov{Z}^{\red}_\la\sub \ov{Z}_\la=\ov{V}_\la\times_{\ov{V}} V$$
defined as in \eqref{ov-Z-la-eq}.
Let us consider the kernels
$$\KK_\lambda = \OO_{\ov{Z}^{\red}_\lambda}$$ 
on $\ov{Z}_\la$.

\begin{theorem}[\protect{\cite[Theorem 6.3.1]{polishchuk-vandenberg-equivariant}}]
  For each \(\lambda\), \(|\lambda| = n\), the functor
  \(\Phi_{\KK_\lambda}:\DD(\ov{V}_\la)\to \DD([V/S_n])\) is fully faithful. The images of these functors
give  a semiorthogonal
	decomposition
	\[
		\mathcal{D}([\mathbb{A}^n/S_n]) = \langle 
		\mathcal{D}(\ov{V}_{\lambda_1}), \ldots,
		\mathcal{D}(\ov{V}_{\lambda_r})\rangle.
	\]
for any total ordering \(\lambda_1<\cdots<\lambda_r\) of
	partitions of \(n\) refining the dominance order.
	\label{thm:pvdb-quotient-sn}
Thus, condition (MSOD) holds for the action of $S_n$ on $\A^n$ and the kernels $(\OO_{\ov{Z}^{\red}_\la})$.
\end{theorem}



For each $\la$, the natural $\G_m$-action on $V_\la$ induces a $\G_m$-action on $\ov{V}_\la$ such that
the morphism $V_\la\to \ov{V}_\la$ is $\G_m$-equivariant. We denote by
\[
\P \ov{V}_\lambda = [(\ov{V}_\lambda\setminus \{0\})/\G_m]
\]  
the corresponding weighted projective space stack.
More precisely, for $\la=(1^{r_1},2^{r_2},\ldots,p^{r_p})$, we get the weighted projective space stack
\[
\P \ov{V}_\lambda=\P(1,\ldots,r_1,1,\ldots,r_2,\ldots,1,\ldots,r_p).
\]
 

\begin{corollary}
	There is a semiorthogonal decomposition
	\[
		\mathcal{D}([\mathbb{P}^{n-1}/S_n])\cong \langle 
		\mathcal{D}(\P \ov{V}_{\lambda_1}), \ldots,
		\mathcal{D}(\P \ov{V}_{\lambda_r})\rangle.
	\]
	\label{cor:pvdb-quotient-descent-pn}
\end{corollary}

\begin{proof}
	This follows from Theorem \ref{thm:pvdb-quotient-sn} by first restricting the semiorthogonal decomposition 
	 to the open subset $\A^n\setminus\{0\}\subset\A^n$ using Corollary \ref{open-res-cor}, and then applying 
	Lemma \ref{passing-to-quot-lem} to the natural $\G_m$-equivariant structures on the corresponding kernels
	(which are the structure sheaves of $\G_m$-invariant correspondences).
	\end{proof}

\subsection{\(S_{n}\)-invariant hypersurfaces}\label{Sn-hypersurface-sec}

Let \(f\in \mathbb{C}[x_{1},\ldots,x_{n}]^{S_{n}}\) be an \(S_{n}\)-invariant
polynomial, and let $H(f)\sub V=\mathbb{A}^n$ be the corresponding hypersurface.

\begin{corollary} 
  Let $T\sub V$ be an $S_n$-invariant closed subset containing the
  singular locus of $H(f)$.  Assume that for every partition $\la$, such that
  $V_\la\setminus T\neq\emptyset$, the restriction $f|_{V_\la}$ is not
  identicially zero and the intersection $H(f|_{V_\la})\cap V_\la^{fr}\setminus T$ is dense in $H(f|_{V_\la})\setminus T$.
  Then the pair $(H(f)\setminus T,S_n)$ satisfies (MSOD), so we have a semiorthogonal decomposition
  $$\DD([H(f)\setminus T/S_n])=\lan \DD((H(f|_{V_{\la_1}})\setminus T)/W_{\la_1}),\ldots,\DD((H(f|_{V_{\la_r}})\setminus T)/W_{\la_r})
  \ran$$
where $\la_1<\ldots<\la_r$ is a total order on partitions of $n$ refining the dominance order.
  \label{thm:sn-inv-divisor}
\end{corollary}

\begin{proof}
First, we use Corollary \ref{open-res-cor} to prove that condition (MSOD), that holds for the pair $(V,S_n)$
by Theorem \ref{thm:pvdb-quotient-sn}, is inherited by the pair $(V\setminus T,S_n)$. Next, we want to check condition ($\ast$) 
for the divisor $H(f)\setminus T$ in $V\setminus T$. Thus, we need to check that
for every partition $\la$ such that $(V\setminus T)_\la=V_\la\setminus T$ is nonempty, 
we have that $H(f)\setminus T$ does not contain $V_\la\setminus T$,
and the intersection $H(f)\cap V_\la\setminus T=H(f|_{V_\la})\setminus T$ contains a dense open subset on which $W_\la$ acts freely.
But this follows from our assumption.
Hence, we can apply Theorem \ref{SOD+thm} to the smooth divisor $H(f)\setminus T$ in $V\setminus T$
to deduce the result. 
\end{proof}


The following simple observation will be useful for us.

\begin{lemma}\label{f-111-lem} For a homogeneous $S_n$-invariant polynomial $f$, such that the corresponding hypersurface $\P H(f)\sub \P^{n-1}$ is smooth, one has
  \begin{equation}\label{f-11-nonvan-eq}	
    f(1,1,\ldots,1)\neq 0. 
   \end{equation}
For any partition $\la$ of $n$, the restriction \(f_{\lambda}:=f|_{V_\la}\) is not identically zero and $H(f_\la)$ is smooth away from the origin.
\end{lemma}  

\begin{proof}
  Indeed, consider the morphism 
  $$\sigma\colon \mathbb{P}^{n-1}\to \mathbb{P}(1,2,\ldots,n)$$
  given by the elementary symmetric polynomials. Then the differential of $\si$
  vanishes identically at the $S_n$-invariant point $(1:1:\ldots:1)$.
  But \(H(f)\) is the preimage of a hypersurface under \(\sigma\). Since
  \(\mathbb{P}(H(f))\) is smooth, we deduce that \( f(1,\ldots,1)\neq 0\).

Since $(1,\ldots,1)\in V_\la$, from \eqref{f-11-nonvan-eq} we get that the restriction
\(f_{\lambda}:=f|_{V_\la}\) is not identically zero. Furthermore,  
$H(f_\la)\setminus\{0\}$ is the fixed locus of a permutation acting on $H(f)\setminus\{0\}$, hence, it is smooth.

\end{proof}

Now we are ready to prove Theorem \ref{Sn-inv-hyper-thm}. Recall that in this theorem we assume that
$f$ is a {\it homogeneous} $S_n$-invariant polynomial such that $\P H(f)$ is smooth.

\begin{proof}[Proof of Theorem \ref{Sn-inv-hyper-thm}]
We would like to apply Corollary \ref{thm:sn-inv-divisor} with $T=\{0\}$ to get condition (MSOD) for the pair
	$(H(f)\setminus \{0\},S_n)$. Let $\la$ be a partition of $n$. 
	By Lemma \ref{f-111-lem}, the restriction
\(f_{\lambda}:=f|_{V_\la}\) is not identically zero. Next, let us check that $H(f_\la)\cap V_\la^{fr}\setminus\{0\}$ is dense
in $H(f_\la)\setminus\{0\}$. In the case $\dim V_\la\le 1$, we have $H(f_\la)\setminus\{0\}=\emptyset$. Thus,
due to the description \eqref{V-la-fr-eq} of $V_\la^{fr}$, 
 it is enough to check that for $\dim V_\la\ge 2$, no component
of the hypersurface $H(f_\la)\sub V_\la$
is contained in a hyperplane $y_i=y_j$ for some $i\neq j$ such that $\la_i=\la_j$.

In the case when the degree of $f$ is $1$, it is proportional to $x_1+\ldots+x_n$, so this is clear.
Now assume that $\deg(f)>1$. By Lemma \ref{f-111-lem}, $H(f_\la)$ is smooth
away from the origin. Thus, if $\dim V_\la\ge 3$ then $H(f_\la)$ is irreducible (since each irreducible component
of $H(f_\la)$ has dimension $\ge 2$ and $H(f_\la)\setminus\{0\}$ is smooth), so it cannot be contained in any
hyperplane. If $\dim V_\la=2$ then $\la$ has only two parts $(\la_1,\la_2)$. In the case $\la_1\neq \la_2$ the statement
is empty, so we only have to check the assertion for $\la=(n/2,n/2)$ assuming
that $n$ is even. But in this case the non-free locus is the line spanned by $(1,\ldots,1)$, so the assertion follows
from \eqref{f-11-nonvan-eq}.
	
Thus, we obtain that the pair $(H(f)\setminus \{0\},S_n)$ satisfies (MSOD).
 It remains to use Lemma \ref{passing-to-quot-lem} to pass to the quotients by $\G_m$.	
\end{proof}

The semiorthogonal decomposition given by Theorem \ref{Sn-inv-hyper-thm} is usually
not motivic since its components are derived categories of some quotient stacks. 
The biggest component of the semiorthogonal decomposition of \(\mathcal{D}([\P H(f)/S_{n}])\) corresponds to the
partition $\la=(1^n)$ and is the image of the pull-back functor with respect to the natural morphism of stacks
$$\pi:[\P H(f)/S_{n}]\to \P H(\ov{f})\sub \P(1,2,\ldots,n),$$
where $\ov{f}$ is $f$ viewed as a quasihomogeneous polynomial on $\A^n/S_n$ (so the target of $\pi$ is the weighted
projective stacky hypersurface). The morphism $\pi$ fits into a Cartesian diagram 
\begin{diagram}
[H(f)\setminus\{0\}/S_n] &\rTo{\wt{\pi}}& H(\ov{f})\setminus\{0\}\\
\dTo{\G_m}&&\dTo{\G_m}\\
[\P H(f)/S_n]&\rTo{\pi}& \P H(\ov{f})
\end{diagram}
in which the vertical arrows are $\G_m$-torsors and the top horizontal arrow is the coarse moduli map for
the action of $S_n$ on $H(f)\setminus \{0\}$.
Note that the fact that the pull-back functor under $\pi$ is fully faithful can be directly deduced from the above diagram.
Indeed, by the projection formula, it is enough to check that $R\pi_*\OO\simeq \OO$. By the base change formula,
this reduces to a similar assertion for the morphism $\wt{\pi}:[H(f)/S_n]\to H(\ov{f})$, which is
the map from a quotient stack by $S_n$ to the corresponding geometric quotient $H(\ov{f})$. For this morphism we have 
$R^{>0}\wt{\pi}_*\OO=0$, and the isomorphism $\wt{\pi}_*\OO\simeq\OO$ follows from the fact that the
algebra of functions $\OO(H(\ov{f}))$ is identified with the subalgebra of $S_n$-invariants, $\OO(H(f))^{S_n}$.


\begin{example}[\(S_3\)-invariant plane curves]
  Let \(C = \mathbb{P}H(f)\subset \mathbb{P}^2\) be an \(S_3\)-invariant smooth plane curve of degree
  \(d\). Since in this case $f(1,1,1)\neq 0$ by Lemma \ref{f-111-lem},
   for the partition \((3)\), we get \(H(\bar{f}_{(3)}) = \{0\}\).
  Hence, the corresponding component in the semiorthogonal decomposition of
  \(\mathcal{D}([C/S_3])\) is zero and can be skipped. Let us consider 
  contributions of the two remaining partitions, \((1^3)\) and
  \((2,1)\).
  \begin{itemize}
    \item[\((1^3)\):] We have identifications \(V_{(1^3)} = V\),
      \(\ov{V}_{(1^3)}\cong \mathbb{A}^3_{1,2,3}\), where the subscripts indicate the
      \(\mathbb{G}_m\)-weights. The vanishing locus of \(\bar{f}_{(1^3)}\),
      \(\P H(\bar{f}_{(1^3)})\) will give a smooth stacky curve in
      \(\mathbb{P}(1,2,3)\).

    \item[\((2,1)\):] We have identifications \(\ov{V}_{(2,1)}=V_{(2,1)} =
      \{y=z\}\subset V\), and \(f_{(2,1)}\) is the restriction of \(f\) to
      this plane. Since \(H(f_{(2,1)})\) is smooth away from the origin, it
      is the union of \(d\) lines through the origin, say \(l_1,\ldots,l_d\).
      The projectivization is the union of \(d\) distinct points
      \(p_1,\ldots,p_d\) in the projective line $\P\ov{V}_{2,1}$.
  \end{itemize}
 Thus, we have a semiorthogonal decomposition
\begin{equation}\label{C/S3-dec}
\DD([C/S_3])=\lan \DD(\P H(\bar{f}_{(1^3)})), \DD(p_1),\ldots,\DD(p_d)\ran
\end{equation}
  
  In the case $d=3$, i.e., when $C$ is an elliptic curve, we can be even more precise about the piece corresponding to
  $(1^3)$. Namely, in this case
   \[
    f(x,y,z) =\alpha {\rm e}_1^3+ \beta {\rm e}_1{\rm e}_2 +\ga {\rm e}_3,
  \]
where ${\rm e}_1,{\rm e}_2,{\rm e}_3$ are elementary symmetric functions in $x,y,z$. Furthermore, we have $\ga\neq 0$ 
(otherwise, $C$ would contain the line ${\rm e}_1=0$). Thus, 
 the equation $f=0$ gives a way to express ${\rm e}_3$ in terms of ${\rm e}_1$ and ${\rm e}_2$. 
Hence,  \(\mathbb{P}H(\bar{f}_{(1^3)})\)  is the weighted projective line stack $\P(1,2)$.

In general, the derived category of \(\mathbb{P}H(\bar{f}_{(1^3)})\) has a
semiorthogonal decomposition with the main component given by the derived category
of the coarse moduli, which is $C/S_3$, and some exceptional objects supported at the stacky points. Thus, the semiorthogonal decomposition 
\eqref{C/S3-dec} can be refined to a decomposition with the main component $\DD(C/S_3)$ followed by exceptional objects.
The obtained decomposition of $\DD([C/S_3])$ matches the one
constructed in \cite{P-toric-orbifold} since the special fibers of the
projection $C\to C/S_3$ are either orbits of the points $p_1,\ldots,p_d$,
corresponding to $\la=(2,1)$, or the points of $C$ mapping to the two stacky points
of $\P(1,2,3)$.

Note that if $d<6$ then the geometric quotient of
\(\mathbb{P}H(\bar{f}_{(1^3)})\) is rational, so in this case
the category \(\mathcal{D}([C/S_3])\) has a full exceptional collection.
\end{example}

Some features of the above example occur in a more general situation.
Below we use the power sum polynomials
$${\rm p}_i(x_1,\ldots,x_n)=x_1^i+\ldots+x_n^i.$$ 

\begin{proposition} Let $f(x_1,\ldots,x_n)$ 
be a generic $S_n$-invariant homogeneous polynomial of degree $d>0$.

\noindent
(i) Let $\la$ be a partition of $n$ such that all parts of $\la$ are distinct.
Then the stack $[\P H(\ov{f}_\la)]$ is actually a smooth projective variety.

\noindent
(ii) Now assume that $\la$ has one part of multiplicity $2$ and all the other parts have multiplicity $1$. 
Then the same conclusion as in (i) holds provided the degree $d$ is even.
\label{nonstacky-pieces-prop}
\end{proposition}

\begin{proof}
(i) For a generic $S_n$-invariant $f$, the hypersurface $H(f)\sub V$ is smooth away from the origin,
hence, the same is true for $H(f_\la)$, the fixed locus of a permutation acting on $H(f)$.
Since $\P\ov{V}_\la$ is the usual projective space, the assertion follows.

\noindent
(ii) We have coordinates $(x,y;z_1,\ldots,z_p)$ on $V_\la$, so
that the embedding $\iota_\la:V_\la\hra V$ has form
$$\iota_\la:(x,y;z_1,\ldots,z_p)\mapsto (x,y,\ldots,x,y;z_1,\ldots,z_1,\ldots,z_p,\ldots,z_p),$$
where $(x,y)$ is repeated $l$ times, each $z_j$ is repeated $m_j$ times,
so that $(l,m_1,\ldots,m_p)$ are all the distinct parts of $\la$, and $l$ (resp., $m_j$)
occur with multiplicity $2$ (resp., $1$) in $\la$.
Set ${\rm p}_1=x+y$, ${\rm p}_2=x^2+y^2$, so that $({\rm p}_1,{\rm p}_2),(z_j)$ are the coordinates on
$\ov{V}_\la$. It is enough to check that $\P H(\ov{f}_\la)$ does not contain stacky points of $\P \ov{V}_\la$,
i.e., the points with ${\rm p}_1=0$ and all $z_j=0$.
Thus, it is enough that $f_\la$ does not vanish at the point of $V_\la$ with $x=-y=1$ and $z_j=0$. 
Note that 
$${\rm p}_2(x_1,\ldots,x_n)|_{\iota_\la(1,-1;0,\ldots,0)}\neq 0.$$ 
Therefore, the same is true for any power of ${\rm p}_2$, and hence,
for a generic $S_n$-invariant polynomial of even degree.
\end{proof}

In the next proposition we show that some components of the semiorthogonal decomposition
of Theorem \ref{Sn-inv-hyper-thm} are derived categories of weighted projective space stacks.

\begin{proposition} 
  Let $f(x_1,\ldots,x_n)$ be an $S_n$-invariant homogeneous polynomial of degree $d\le n$
  such that $\P H(f)$ is smooth, and such that in the expression of $f$ as a polynomial in ${\rm p}_1,\ldots,{\rm p}_n$
the coefficient of ${\rm p}_d$ is nonzero. 
  Then for a partition $\la=(1^{r_1},2^{r_2},\ldots,p^{r_p})$, such that
$r_l\ge d$ for some $l$, the stack $[\P H(\ov{f}_\la)]$ is isomorphic to
  the weighted projective space stack with the weights obtained by removing one weight $d$ from the sequence
  $$(1,\ldots,r_1,1,\ldots,r_2,\ldots,1,\ldots,r_p).$$
\label{cubic-mult3-prop}
\end{proposition}

\begin{proof} 
 Let $y_1,\ldots,y_{r_l},z_1,\ldots,z_N$ be the coordinates on
  $V_\la$, where $y_1,\ldots,y_{r_l}$ are the coordinates corresponding to the parts of $\la$ equal to $l$,
  so that in the embedding $V_\la\hra V$ each of
  coordinates $y_1,\ldots,y_{r_l}$ is repeated $l$ times (see \eqref{V-la-partition-eq}). 
  Note that the coordinates on $\ov{V}_\la$ are given by 
  the functions $({\rm p}_i(y_1,\ldots,y_{r_l}))_{1\le i\le r_l}$, as well as some symmetric functions in other groups of variables.
  Since $r_l\ge d$, ${\rm p}_d(y_1,\ldots,y_{r_l})$ is one of the coordinates on $\ov{V}_\la$. 

  It suffices to check that ${\rm p}_d(y_1,\ldots,y_{r_l})$ occurs with nonzero coefficient in $f_\la$. Indeed,
then we can use $f_\la$ to express the coordinate \({\rm p}_d(y_1,\ldots,y_{r_l})\) in terms of other coordinates on $\ov{V}_\la$,
which gives our assertion. We have
$$f(x_1,\ldots,x_n)=\a\cdot {\rm p}_d(x_1,\ldots,x_n)+g(x_1,\ldots,x_n),$$
where $g$ is a polynomial in $({\rm p}_i(x_1,\ldots,x_n))_{1\le i<d}$. Hence, the restriction of $g$ to $V_\la$ is expressed in
terms of coordinates of weight $<d$ on $\ov{V}_\la$, so it does not contribute to 
the coefficient of ${\rm p}_d(y_1,\ldots,y_{r_l})$. Furthermore, we have
 $${\rm p}_d(\underbrace{y_1,\ldots,y_1}_l,\ldots,\underbrace{y_{r_l},\ldots,y_{r_l}}_{l},z_1,\ldots)=
 l\cdot {\rm p}_d(y_1,\ldots,y_{r_l}) \ \mod (\C z_1^d+\ldots+\C z_N^d).$$ 
Hence, the coefficient of ${\rm p}_d(y_1,\ldots,y_{r_l})$ in $f_\la$ is equal to $l\cdot \a$. In particular, this
  coefficient is nonzero, as required.
\end{proof}

\begin{corollary}\label{cubic-mult3-cor} 
The conclusion of Proposition \ref{cubic-mult3-prop} holds for any
$S_n$-invariant homogeneous polynomial $f$ of degree $d\le 3$ such that $\P H(f)$ is smooth.
\end{corollary}

\begin{proof}
The case $d=1$ is trivial. In the case $d=2$, we have
$$f=\a {\rm p}_2+\b {\rm p}_1^2,$$
while in the case $d=3$, we have
$$f=\a {\rm p}_3+\b {\rm p}_1{\rm p}_2+\ga {\rm p}_1^3,$$ 
In both cases $\a\neq 0$, since otherwise $f$ would be reducible.  
Hence, we can apply Proposition \ref{cubic-mult3-prop}.
\end{proof}


In the case of cubic forms in $\le 6$ variables, we obtain from Theorem
\ref{Sn-inv-hyper-thm} the following decompositions of $S_n$-equivariant derived
categories.

\begin{proposition} 
  Let $f(x_1,\ldots,x_n)$ be a generic $S_n$-invariant homogeneous cubic
  polynomial, where $n\le 5$.  Then $\DD([\P H(f)/S_n])$ has a full exceptional
  collection. For $n=6$, there is an exceptional collection in $\DD([\P H(f)/S_6])$ 
  such that its right orthogonal is equivalent to $\DD(E)$, where $E$
  is the elliptic curve given by the cubic $f_{(3,2,1)}$ in $\P
  V_{(3,2,1)}\simeq \P^2$.
\label{cubic-prop}
\end{proposition}

\begin{proof} 
First of all, we observe that for $n\le 5$, a partition $\la$ of $n$ can have at most two distinct parts, while for $n=6$ the only
partition with $3$ distinct parts is $(3,2,1)$.

  By Corollary \ref{cubic-mult3-cor}, if $\la$ has a part of multiplicity
  $\ge 3$ then the corresponding piece in the semiorthogonal decomposition of
  Theorem \ref{Sn-inv-hyper-thm} is the derived category of the weighted
  projective space stack, so it has a full exceptional collection (see \cite[Sec.\ 2]{AKO}).

Now we claim that all partitions with at most two distinct parts, each of mulitplicity at most $2$, lead to subcategories
generated by exceptional collections. We prove this case by case. Note that the partition $\la=(n)$ does not contribute to
the semiorthogonal decomposition since $f(1,\ldots,1)\neq 0$.

\noindent
{\bf Case $\la=(l,l)$}. Then $V_\la$ has coordinates $x,y$ and $\ov{V}_\la$ has coordinates ${\rm p}_1=x+y$, ${\rm p}_2=x^2+y^2$.
The unique stacky point of the weighted projective line stack $\P \ov{V}_\la=\P(1,2)$ is given by ${\rm p}_1=0$.
Note that ${\rm p}_3=x^3+y^3$ is divisible by ${\rm p}_1$, so $\ov{f}_\la$ vanishes at this point. It follows that
$\P H(\ov{f}_\la)$ is the union of one point and of one stacky point with the automorphism group $\Z/2$.
The derived category of such stacky point splits as the direct sum of two derived categories of the usual point.

\noindent
{\bf Case $\la=(l_1,l_2)$, where $l_1>l_2$}. Then $f_\la$ is a cubic on the $2$-dimensional space
$V_\la$, with isolated singularity at the origin, so $\P H(f_\la)$ is the union of three distinct points.

\noindent
{\bf Case $\la=(l,1,1)$ with $l>1$ or $\la=(2,2,1)$}. Then $V_\la$ has coordinates $x,y,z$, where $W_\la=S_2$
swaps $x$ and $y$, so that $\ov{V}_\la$ has coordinates ${\rm p}_1=x+y$, ${\rm p}_2=x^2+y^2$ and $z$, and $\P\ov{V}_\la=\P(1,1,2)$.
The cubic $\ov{f}_\la$ should have form 
$$\ov{f}_\la={\rm p}_2(\a z+\b {\rm p}_1)+C({\rm p}_1,z),$$ 
where $C({\rm p}_1,z)$ is a binary cubic form.
It is easy to see that for generic $S_n$-invariant $f$,
one has $\a\neq 0$, so we can make the change of variables $z_1=\a z+\b {\rm p}_1$. Furthermore,
$C({\rm p}_1,z)$ is not divisible by $z_1$, since $f_\la$ has an isolated singularity at $0$.
Thus, rescaling the variables, we can bring $f$ to the form
$$\ov{f}_\la={\rm p}_2z_1+z_1Q({\rm p}_1,z_1)+{\rm p}_1^3,$$
where $Q$ is a binary quadratic form. Now taking
$u={\rm p}_2+Q({\rm p}_1,z_1)$ as a new variable of weight $2$, we get
$$\ov{f}_\la=uz_1+{\rm p}_1^3.$$
It is easy to see that $\P H(\ov{f}_\la)$ is isomorphic to the weighted projective line stack $\P(1,2)$.
Namely, there is an isomorphism given by
$$\P(1,2)\to \P H(\ov{f}_\la): (t:v)\mapsto (u=v^3, z_1=-t^3, {\rm p}_1=vt).$$

{\bf Case $\la=(2,2,1,1)$}. 
Then we have coordinates $x_1,y_1,x_2,y_2$ on $V_\la$, and $W_\la=S_2\times S_2$
permutes $x_1$ with $y_1$ and $x_2$ with $y_2$. Set ${\rm p}_1(i)=x_i+y_i$, ${\rm p}_2(i)=x_i^2+y_i^2$. Then
the cubic $\ov{f}_\la$ has form
$$\ov{f}_\la={\rm p}_2(1)z_1+{\rm p}_2(2)z_2+C({\rm p}_1(1),{\rm p}_1(2)),$$
where $z_1$ and $z_2$ are some linear forms in ${\rm p}_1(1),{\rm p}_1(2)$. It is easy to see that for generic $f$, the linear forms $z_1$ and $z_2$ will be linearly independent, so we can view ${\rm p}_2(1),{\rm p}_2(2),z_1,z_2$ as independent variables. Now
adding to ${\rm p}_2(i)$ appropriate quadratic expressions of $z_1,z_2$, we can rewrite $\ov{f}_\la$ as
$$\ov{f}_\la=u_1z_1+u_2z_2,$$
where $u_1,u_2,z_1,z_2$ are independent variables ($\deg(u_i)=2$, $\deg(z_i)=1$).
Thus, we can identify $\P H(\ov{f}_\la)$ with $\P(1,2)\times \P^1$ via the isomorphism
$\P(1,2)\times \P^1\to \P H(\ov{f}_\la)$,
$$ (t:v), (s_1:s_2)\mapsto (u_1=vs_1, u_2=vs_2, z_1=ts_2, z_2=-ts_1).$$

Thus, for $n\le 6$, all of the subcategories corresponding to $\la\neq (3,2,1)$
admit full exceptional collections. The remaining subcategory corresponding to $\la=(3,2,1)$ (for $n=6$)
is equivalent to $\DD(E)$, where $E$ is the elliptic curve given by $f_{(3,2,1)}$.
\end{proof}

\begin{remark} Using Corollary \ref{cubic-mult3-cor} we see that for an $S_n$-invariant nondegenerate quadric $f$, the components of the semiorthogonal decomposition
of Theorem \ref{Sn-inv-hyper-thm} are either smooth projective quadrics (for partitions with distinct parts)
or weighted projective space stacks (for the remaining partitions). In particular, in this case
$\DD([\P H(f)/S_n])$ has a full exceptional collection.
\end{remark}


\subsection{Products of curves}\label{other-ex-sec}

First, let us consider the case of an action of a finite group $G$ on a smooth curve (we assume that the action is
effective).
Note that in this case the quotient $C/G$ is smooth
(since passing to invariants of a finite group preserves normality)
and the stabilizer subgroup $\St_x$ of every point $x\in C$ is cyclic. Let
$$R=D_1\sqcup\ldots\sqcup D_r\sub C$$ 
be the decomposition into $G$-orbits of the ramification locus of the projection
$\pi: C\to C/G$. Then each $D_i$ is a fiber of $\pi$ and the stabilizer of a point in $D_i$ is isomorphic to $\Z/m_i\Z$.
Then the proof of \cite[Thm.\ 1.2]{P-toric-orbifold}
implies that for each $i$, there is an exceptional collection of $G$-equivariant sheaves on $C$,
\begin{equation}\label{divisor-exc-collection}
(\om_C|_{D_i},\om^{\ot 2}_C|_{D_i},\ldots,\om_C^{\ot m_i-1}|_{D_i}),
\end{equation}
and if ${\mathcal A}_i\sub \DD([C/G])$ is the subcategory generated by this collection, then there is a semiorthogonal decomposition
\begin{equation}\label{curve-decomposition-eq}
\DD([C/Y])=\lan {\mathcal A}_1,\ldots,{\mathcal A}_r, \pi^*\DD(C/G)\ran,
\end{equation}
where $\pi^*:\DD(C/G)\to \DD([C/G])$ is the pull-back functor.
More precisely, in \cite{P-toric-orbifold} a different decomposition was considered,
\begin{equation}\label{curve-decomposition-bis}
\DD([C/Y])=\lan \pi^*\DD(C/G), \BB_1,\ldots,\BB_r \ran,
\end{equation}
with ${\mathcal A}_i=\om_C\ot \BB_i$, from which \eqref{curve-decomposition-eq} is obtained by mutation. 
Also, in \cite{P-toric-orbifold} it was shown that each $\BB_i$
is generated by the exceptional collection
$$(\OO_{(m_i-1)D_i},\ldots,\OO_{2D_i},\OO_{D_i}),$$
which can be mutated into $(\OO_{D_i},\OO(-D_i)|_{D_i},\ldots,\OO(-(m_i-2))|_{D_i})$.
The collection \eqref{divisor-exc-collection} in ${\mathcal A}_i$ is obtained from the latter collection by tensoring with $\om_C$.

This leads to the following result.

\begin{proposition}\label{curve-SOD+prop} 
Let $G$ be a finite group acting effectively on a smooth curve $C$. Then condition (MSOD) is satisfied, 
where the kernel corresponding to $\la=1$ is the structure sheaf of the graph of $\pi$.
\end{proposition}

\begin{proof} We claim that the semiorthogonal decomposition \eqref{curve-decomposition-eq} 
(or \eqref{curve-decomposition-bis}) can be restructured to get
the decomposition required by (MSOD). Namely, \eqref{curve-decomposition-eq} consists of the image of the
pull-back functor $\pi^*:\DD(C/G)\to \DD([C/G])$, along with $m_i-1$ exceptional objects supported on $D_i$, for 
$i=1,\ldots,r$. On the other hand, for (MSOD) to hold, for each conjugacy class representative $g\neq 1$, 
and each $C(g)$-orbit in $C^g$, we need to have one exceptional object in $\DD([C/G])$
supported on the corresponding $G$-orbit in $C$. The fact that the numbers of exceptional objects supported
on each $G$-orbit match was proved in \cite[Rem.\ 4.3.2]{polishchuk-vandenberg-equivariant}.
\end{proof}

Using Lemma \ref{product-lem} we deduce the following

\begin{corollary}\label{product-curves-cor} 
Let $C_1,\ldots,C_n$ be smooth curves, and for each $i$, let $G_i$ be a finite
group acting effectively on $C_i$.  Then condition (MSOD) holds for the action of
$G_1\times\cdots\times G_n$ on $C_1\times\cdots\times C_n$.
\end{corollary}

\begin{example}\label{cyclic-group-action-ex}
  For the standard action of the cyclic group \(\mu_d\) on \(\mathbb{A}^1\), 
  the geometric
  quotient is isomorphic to \(\mathbb{A}^1_d\), where the subscript \(d\) indicates the
  \(\mathbb{G}_m\)-weight, so that the quotient map \(\pi\colon \mathbb{A}^1\to
  \mathbb{A}^1_d\) is given by \(x\mapsto x^d\).
  We have a semiorthogonal decomposition
  \[
    \mathcal{D}([\mathbb{A}^1/\mu_d]) = \langle
    \mathcal{O}_p\otimes\chi^{d-1},\ldots,\mathcal{O}_p\otimes\chi,
    \pi^\ast\mathcal{D}(\mathbb{A}^1_d)\rangle,
  \]
  where \(\mathcal{O}_p\) denotes the structure sheaf of the origin, and
  $\chi:\mu_d\to\G_m$ is the character given by the natural embedding. 
  
  Now, for positive integers \(d_1,\ldots,d_k\), let us consider the natural
  action of \(G = \mu_{d_1}\times\cdots\times \mu_{d_k}\) on $\A^k$ (where the
  $i$th factor acts on the $i$th coordinate).  By Corollary
  \ref{product-curves-cor},  we have a motivic semiorthogonal decomposition of
  \(\mathcal{D}([\mathbb{A}^k/G])\). We can describe explicitly the components of this
  decomposition as follows.  The fixed locus of an element of
  \(g=(z_1,\ldots,z_k)\in G\) is isomorphic to the affine space
  \(\mathbb{A}^{n_g}\), where  \(n_g\) is the number of trivial components of
  \(g\) . The geometric quotient by \(C(g)=G\) is \(\pi_g\colon \mathbb{A}^{n_g}\to
  \mathbb{A}^{n_g}_{\bd_g}\), where \(\bd_g\) is a multi-index giving weights
  for the \(\mathbb{G}_m\)-action (\(\bd_g\) is the set of \(d_i\) for which
  \(z_i=1\)).  Let \(\iota_g\colon \mathbb{A}^{n_g}\hookrightarrow
  \mathbb{A}^k\) denote the natural embedding. Then the composite functor
  \[
    \iota_{g\ast}\circ \pi_g^\ast\colon \mathcal{D}(\mathbb{A}^{n_g}_{\bd_g})\to
    \mathcal{D}([\mathbb{A}^k/G])
  \]
  is fully faithful.

  For each $i$, let $\zeta_{d_i}$ be a $d_i$th primitive root of unity.  For
  $g=(\zeta_1^{m_1},\ldots,\zeta_k^{m_k})\in G$, where $0\le m_i<d_i$, we define
  the character $\chi_g$ of $G$ by setting \(\chi_g =
  \chi_1^{m_1}\cdots\chi_k^{m_k}\), where \(\chi_i\colon G\to \G_m\) is given by
  the \(i\)th projection.
  
  Then the functors giving the semiorhogonal decomposition of $\DD([\A^k/G])$
  (numbered by $g\in G$) are 
  \[
    (\iota_{g\ast}\circ\pi_g^\ast)\otimes\chi_g\colon \mathcal{D}(\mathbb{A}^{n_g}_{\bd_g})\to
    \mathcal{D}([\mathbb{A}^k/G]),
  \]
  ordered lexicographically with respect to the reverse order on each set
  $\{0,\ldots,d_i-1\}$. Thus, we have a semiorthogonal decomposition
  \[
    \DD([\mathbb{A}^k/G]) = \langle
    \DD(pt)\otimes\chi_1^{d_1-1}\cdots\chi_k^{d_k-1},\ldots,
    \pi^\ast\DD(\mathbb{A}^k_{d_1,\ldots,d_k})\rangle.
  \]
    
  As before, we can delete the origin in all the affine spaces and pass to
  \(\mathbb{G}_m\)-equivariant categories. In this way we get a semiorthogonal
  decomposition of \(\mathcal{D}([\mathbb{P}^{k-1}/G])\) indexed by the elements
  of \(G\). The components of this semiorthogonal decomposition will be the weighted
  projective space stacks \(\mathbb{P}(\bd_g)\). 
  
  We can also apply Theorem \ref{SOD+thm} to get, as in Section
  \ref{Sn-hypersurface-sec}, a semiorthogonal decomposition of $\DD([\P H(f)/G])$,
  where $f$ is a $G$-invariant homogeneous polynomial on $\A^k$.  More
  precisely, we have to assume that $\P H(f)$ is smooth and that restrictions of
  $f$ to certain coordinate subspaces are nonzero. Namely, in the case when
  there exists an index $i$ with $d_i=1$, we have to assume the
  nonvanishing of the restriction of $f$ to the subspace where all coordinates
  with $d_i>1$ are set to zero. In the case when all $d_i>1$, we have to assume
  that the restriction of $f$ to each coordinate line is nonzero.
 
  For example, if $d_1>1$, $d_2=\ldots=d_k=1$, and
  $f=x_1^{d_1}-g(x_2,\ldots,x_k)$, then $\P H(f)$ is a cyclic cover of
  $\P^{k-2}$ and our decomposition of $\DD([\P H(f)/\mu_{d_1}])$ matches the one
  given by Kuznetsov-Perry in \cite[Theorem 4.1]{kuz-perry-17}.
\end{example}
  
\bibliographystyle{amsalpha}
\bibliography{invariant-divisor-sod}

\def\cprime{$'$}
\providecommand{\bysame}{\leavevmode\hbox to3em{\hrulefill}\thinspace}
\providecommand{\MR}{\relax\ifhmode\unskip\space\fi MR }
\providecommand{\MRhref}[2]{%
  \href{http://www.ams.org/mathscinet-getitem?mr=#1}{#2}
}
\providecommand{\href}[2]{#2}
\begin{thebibliography}{BZNP17}

\bibitem[AKO08]{AKO}
Denis Auroux, Ludmil Katzarkov, and Dmitri Orlov, \emph{Mirror symmetry for
  weighted projective planes and their noncommutative deformations}, Ann. of
  Math. (2) \textbf{167} (2008), no.~3, 867--943. \MR{2415388}

\bibitem[BGLL17]{BGLL-cat-measures-17}
Daniel {Bergh}, Sergey {Gorchinskiy}, Michael {Larsen}, and Valery {Lunts},
  \emph{{Categorical measures for finite group actions}}, arXiv e-prints
  (2017), arXiv:1709.00620.

\bibitem[BO02]{bondal-orlov-sod}
A.~Bondal and D.~Orlov, \emph{Derived categories of coherent sheaves},
  Proceedings of the {I}nternational {C}ongress of {M}athematicians, {V}ol.
  {II} ({B}eijing, 2002), Higher Ed. Press, Beijing, 2002, pp.~47--56.
  \MR{1957019}

\bibitem[Bri06]{bridgeland-derived}
Tom Bridgeland, \emph{Derived categories of coherent sheaves}, International
  {C}ongress of {M}athematicians. {V}ol. {II}, Eur. Math. Soc., Z\"urich, 2006,
  pp.~563--582. \MR{2275610}

\bibitem[BS17]{Bergh-Schn}
Daniel {Bergh} and Olaf~M. {Schn{\"u}rer}, \emph{{Conservative descent for
  semi-orthogonal decompositions}}, arXiv e-prints (2017), arXiv:1712.06845.

\bibitem[BZNP17]{bn}
David Ben-Zvi, David Nadler, and Anatoly Preygel, \emph{Integral transforms for
  coherent sheaves}, J. Eur. Math. Soc. (JEMS) \textbf{19} (2017), no.~12,
  3763--3812. \MR{3730514}

\bibitem[Ela12]{elagin:descent-sod}
A.~D. Elagin, \emph{Descent theory for semi-orthogonal decompositions}, Mat.
  Sb. \textbf{203} (2012), no.~5, 33--64. \MR{2976858}

\bibitem[Gro67]{ega4-4}
A.~Grothendieck, \emph{{\'El\'ements de g\'eom\'etrie alg\'ebrique. {IV}.
  \'Etude locale des sch\'emas et des morphismes de sch\'emas {IV}, Quatri\'eme
  partie}}, Inst. Hautes \'Etudes Sci. Publ. Math. (1967), no.~32, 361.
  \MR{0238860}

\bibitem[Huy06]{huybrechts-fourier}
D.~Huybrechts, \emph{Fourier-{M}ukai transforms in algebraic geometry}, Oxford
  Mathematical Monographs, The Clarendon Press, Oxford University Press,
  Oxford, 2006. \MR{2244106}

\bibitem[KP17]{kuz-perry-17}
Alexander Kuznetsov and Alexander Perry, \emph{Derived categories of cyclic
  covers and their branch divisors}, Selecta Math. (N.S.) \textbf{23} (2017),
  no.~1, 389--423. \MR{3595897}

\bibitem[Kre09]{kresch:geometry-dm}
Andrew Kresch, \emph{On the geometry of {D}eligne-{M}umford stacks}, Algebraic
  geometry---{S}eattle 2005. {P}art 1, Proc. Sympos. Pure Math., vol.~80, Amer.
  Math. Soc., Providence, RI, 2009, pp.~259--271. \MR{2483938}

\bibitem[Kuz06]{kuz:hyperplane-section}
Alexander Kuznetsov, \emph{Hyperplane sections and derived categories}, Izv.
  Ross. Akad. Nauk Ser. Mat. \textbf{70} (2006), no.~3, 23--128. \MR{2238172}

\bibitem[Kuz11]{kuz:sod-base-change}
\bysame, \emph{Base change for semiorthogonal decompositions}, Compos. Math.
  \textbf{147} (2011), no.~3, 852--876. \MR{2801403}

\bibitem[{Lim}16]{L-sum-potentials}
Bronson {Lim}, \emph{{Equivariant Derived Categories Associated to a Sum of Two
  Potentials}}, arXiv e-prints (2016), arXiv:1611.03058.

\bibitem[Lun73]{luna-slices-etale}
Domingo Luna, \emph{Slices \'etales}, 81--105. Bull. Soc. Math. France, Paris,
  M\'emoire 33. \MR{0342523}

\bibitem[Mum08]{Mum-ab-var}
David Mumford, \emph{Abelian varieties}, Tata Institute of Fundamental Research
  Studies in Mathematics, vol.~5, Published for the Tata Institute of
  Fundamental Research, Bombay; by Hindustan Book Agency, New Delhi, 2008, With
  appendices by C. P. Ramanujam and Yuri Manin, Corrected reprint of the second
  (1974) edition. \MR{2514037}

\bibitem[Pol06]{P-toric-orbifold}
A.~Polishchuk, \emph{Holomorphic bundles on 2-dimensional noncommutative toric
  orbifolds}, Noncommutative geometry and number theory, Aspects Math., E37,
  Friedr. Vieweg, Wiesbaden, 2006, pp.~341--359. \MR{2327312}

\bibitem[PV15]{polishchuk-vandenberg-equivariant}
Alexander {Polishchuk} and Michel {Van den Bergh}, \emph{{Semiorthogonal
  decompositions of the categories of equivariant coherent sheaves for some
  reflection groups}}, arXiv e-prints (2015), arXiv:1503.04160.

\bibitem[TVdB18]{tabuada-vandenberg-additive}
Gon\c{c}alo Tabuada and Michel Van~den Bergh, \emph{Additive invariants of
  orbifolds}, Geom. Topol. \textbf{22} (2018), no.~5, 3003--3048. \MR{3811776}

\end{thebibliography}

\end{document}